\documentclass{amsart}
\usepackage{amssymb}

\newtheorem{thm}{Theorem}[section]
\newtheorem{lem}[thm]{Lemma}

\newtheorem{cor}[thm]{Corollary}
\newtheorem{prop}[thm]{Proposition}

\theoremstyle{definition}  
\newtheorem{example}[thm]{Example}
\newtheorem{definition}[thm]{Definition}
\newtheorem{notation}[thm]{Notation}

\numberwithin{equation}{section}

\usepackage{pdflscape}

\usepackage{pgf}
\usepackage{tikz}
\usetikzlibrary{arrows,automata}
\usepackage[latin1]{inputenc}
\usepackage{verbatim}

\newcommand{\add}{(+1)}
\newcommand{\Z}{\mathbb Z}
\newcommand{\N}{\mathbb N}

\DeclareMathOperator{\Stab}{Stab}
\DeclareMathOperator{\Aut}{Aut}
\DeclareMathOperator{\Id}{Id}

\usepackage{color}

\newcommand{\Aa}{{\mathcal A}}

\newcommand{\Gg}{{\mathcal G}}

\newcommand{\Rr}{{\mathcal R}}

\newcommand{\Tt}{{\mathcal T}}

\newcommand{\id}{{\mathbf 1}}

\newcommand{\one}{{\mathbf 1}}

\newcommand{\tc}{G^{(2)}}

\newcommand{\Xmax}{X_{max}}

\newcommand{\pmax}{\pi_{max}}
\newcommand{\M}{\mathcal M}

\newcommand{\rst}{right topological}

\newcommand{\rset}{I_\theta}
\newcommand{\Zl}{{ \Z_\ell}}
\newcommand{\Ef}{E^{fib}}
\newcommand{ \Sstr}{structural semigroup}
\newcommand{\fg}{\mathfrak f}
\newcommand{\Cf}{\mathrm{Cov}}
\newcommand{\Ct}{\Cf_{\Et}}
\newcommand{\Et}{\Tt}
\newcommand{\Gstr}{G_\theta}
\newcommand{\Glstr}{\overline{\Gamma}_\theta}

\newcommand{\Sfib}{\M^{fib}_0(X_\theta)}
\newcommand{\RSTp}{\mathcal G}
\newcommand{\RSTfp}{\RSTp^{fib}}
\newcommand{\RST}{\RSTp_\theta}
\newcommand{\RSTf}{{\RST^{fib}}}

\newcommand{\evo}{{\mathfrak e}{\mathfrak v}_0}

\thanks{J.K.\ would like to thank Marcy Barge for discussions about the Ellis semigroup of substitutional systems. Both authors thank Eli Glasner and Olivier Mathieu for their useful comments. This project has received funding from the European Research Council (ERC) under the European Union's Horizon 2020 
research and innovation programme under the Grant Agreement No 648132. }

\begin{document}
\title[The Ellis semigroup of bijective substitutions]{The Ellis semigroup of bijective substitutions}

\author{Johannes Kellendonk}
\address{
	 Institut Camille Jordan, Universit\'{e} Lyon-1, France
	}
\email{kellendonk@math.univ-lyon1.fr}

\author{Reem Yassawi}
\address{ School of Mathematics and Statistics, The Open University, U.K. and Institut Camille Jordan, Universit\'{e} Lyon-1, France 
}
\email{reem.yassawi@open.ac.uk}


\begin{abstract}
For topological dynamical systems $(X,T,\sigma)$ with abelian group $T$ which admit an equicontinuous factor 
$\pi:(X,T,\sigma)\to (Y,T,\delta)$ the Ellis semigroup $E(X)$ is an extension of $Y$ by its subsemigroup $ \Ef(X)$ of elements which preserve the fibres of $\pi$.
We establish methods to compute $\Ef(X)$ and use them to determine the Ellis semigroup 
of dynamical systems arising from primitive aperiodic bijective substitutions. 
As an application we show that for these substitution shifts, the virtual automorphism group is isomorphic to the classical automorphism group.
\end{abstract}

\maketitle

\section{Introduction}

Consider $(X,T,\sigma)$, the action $\sigma$ of an abelian group (or semigroup) $T$ by homeomorphisms on a compact  metrisable space $X$. 
The {\em Ellis semigroup} $E(X)$ of   $(X,T,\sigma)$ is {the} compactification of the group action in the topology of pointwise convergence on  $X^X$. 
The study of its topological and algebraic structure, which  was initiated by Ellis \cite{Ellis-1960}, reveals dynamical properties of $(X,T,\sigma)$ and is 
consequently an area of active study. 

One topological property which has recently incited a lot of interest is tameness:   $(X,T,\sigma)$ is {\em tame} if $E(X)$ is the sequential compactification of the action \cite{Huang-2006, glasnerCM}, that is, each element of $E(X)$ is a limit of a sequence (as opposed to a limit of a net, or generalised sequence) of homeomorphisms coming from the group action. 
This can be expressed purely using cardinality: $E(X)$ is tame if and only if its cardinality
is at most that of the continuum 
 \cite{glasnerCM}. 
 Tameness implies, for instance, the following dynamical property \cite{Huang-2006,glasner2018structure,fuhrmann2018irregular}:  
If a compact metrisable minimal system which admits an invariant measure is tame, then it is a {\em $\mu$-almost one to one} extension of its maximal equicontinuous factor. Here $\mu$ is the unique ergodic probability measure on the maximal equicontinuous factor of $(X,T,\sigma)$, and $\mu$-almost one-to-one means that the set of points in the maximal equicontinuous factor which have a unique pre-image under the factor map has full $\mu$-measure. As soon as all fibres of the maximal equicontinuous factor map contain more than one point, the system is thus not tame.

Systematic investigations focussing on the algebraic structure of $E(X)$ are to our knowledge, restricted to the question of when $E(X)$ is a group, when it has a single minimal left ideal, or, in the case of $T=\Z^+$, when its adherence subsemigroup is left simple. $E(X)$ is a group if and only if $(X,T,\sigma)$ is {\em distal} (proximality is trivial), 
$E(X)$ has a single minimal left ideal if and only if proximality is transitive (see, for instance, \cite{Auslander}), and the adherence subsemigroup is left simple if and only if forward proximality implies forward asymptoticity \cite{Blanchard}. 
Recently, a detailed computation of 
the Ellis semigroups of the dynamical systems arising from almost canonical projection method tilings \cite{Aujogue,Aujogue-Barge-Kellendonk-Lenz} has exhibited another algebraic structure which seems worthwhile investigating, namely 
the semigroups are all disjoint unions of groups. Semigroups which are disjoint unions of groups are precisely those which are {\em completely regular}, which means  that every element admits a generalised inverse with which it commutes. Ellis semigroups associated to almost canonical projection method tilings are tame \cite{Aujogue}. 

For the most part, good descriptions of Ellis semigroups are only currently available for tame systems.
The present paper arose from a desire to obtain explicit algebraic descriptions of Ellis semigroups for a class of dynamical systems which are not tame. We study the Ellis semigroup of systems $(X_\theta, \Z,\sigma)$ arising from  bijective substitutions 
$\theta$. The fibres of the maximal equicontinuous factor map of such systems are never singletons and so the resulting semigroup is not tame. They also enjoy two properties which we harness. The first is that for these $\Z$-actions, forward and backward  proximality are non-trivial and equal to forward and backward asymptoticity. We describe systems $(X,\Z,\sigma)$ with this property in Section \ref{Section:comp reg}, and show that their Ellis group $E(X)$ is the disjoint union of the acting group $\Z$ with its {\em kernel} $\M(X)$, that is, the smallest bilateral ideal of $E(X)$; in particular, $E(X)$ is  completely regular.
This reduces the task to the study of $\M(X)$.

The kernel $\M(X)$ of a compact sub-semigroup of $X^X$ is always {\em completely simple} and therefore can be described by the Rees-Suskevitch  theorem and its topological extensions (Theorems \ref{Rees-theorem}, \ref{thm-Rees-str-1}, \ref{thm-Rees-str-2}). 
This theorem characterises a completely simple semigroup  as a {\em matrix semigroup} $M[G;I,\Lambda;A]$, where $G$ is the so-called {\em structure group}, where $I$ and $\Lambda$ index the right and left ideals respectively, and where $A$ is a matrix through which the semigroup  operation is defined. Its entries are elements of $G$ which specify the idempotents. 

The second property that bijective substitution systems enjoy is that they  are {\em unique singular orbit systems}. This means that they have exactly one orbit of singular fibres 
(fibres of the factor map on which proximality is non-trivial) over an equicontinuous factor.  We study these systems in Section~\ref{fibre-preserving} in a way which can be summarised as follows.
Given an  equicontinuous factor $\pi:(X,T,\sigma)\to (Y,T,\delta)$ we obtain a
short exact sequence of right-topological semigroups for the Ellis semigroup which restricts to a short exact sequence of its kernel
\begin{equation}\label{eqn:exact}
 \Ef(X) \hookrightarrow E(X) \stackrel{\tilde\pi} \twoheadrightarrow E(Y)\cong Y \quad \mathrm{and}\quad  \M^{fib}(X) \hookrightarrow \M(X) \stackrel{\tilde\pi} \twoheadrightarrow \M(Y)\cong Y.
\end{equation}
Here $\Ef(X)$ is the subsemigroup of functions which preserve the fibres of the factor map $\pi$, and  $\M^{fib}(X)$ is the kernel of $\Ef(X)$. The two matrix semigroups associated to  $\M^{fib}(X)$ and $\M(X)$ via the Rees-Suskevitch  theorem are related (when properly normalised): they share the same $I$, $\Lambda$ and $A$, and their corresponding structure groups form an exact sequence
\begin{equation}\label{eqn:exact-1}
  \RSTfp \hookrightarrow \RSTp \stackrel{\tilde\pi} \twoheadrightarrow  Y,
\end{equation} 
derived from the above (\ref{eqn:exact}). Finally we make one further reduction: we restrict $\M^{fib}(X)$ to a  singular fibre and obtain again a completely simple semigroup to which we can apply the Rees-Suskevitch  theorem. If $Y$ contains a single $T$-orbit, say that of $y_0$, such that $\pi^{-1}(y_0)$ is singular, then the restriction of  $\M^{fib}(X)$ to $\pi^{-1}(y_0)$, denoted $\M^{fib}_{y_0}(X)$, has a matrix form which shares the same $I$, $\Lambda$ and $A$ as the other two matrix semigroups above. Furthermore, 
we show in 
Corollary~\ref{cor:reproducing M} that, if the singular fibres are finite and the idempotents generate $\M^{fib}_{y_0}(X)$, then the structure group $\mathcal G^{fib}$ equals the infinite Cartesian product $G_\pi^{Y/T}$,  where $G_\pi$ is the structure group of 
$\M^{fib}_{y_0}(X)$, and $Y/T$ is the space of $T$-orbits of $Y$. We thus obtain a description of $\M(X)$ through the finite semigroup $\M^{fib}_{y_0}(X)$ and the extension (\ref{eqn:exact-1}). We prove that the extension is algebraically split so that $\RSTp$ is a semidirect product of $\RSTfp$ with $Y$.
While $\M^{fib}(X)$ is topologically isomorphic its matrix semigroup representation, $\M(X)$ is only algebraically isomorphic to it.  

The dynamical system $(X_\theta,\Z, \sigma)$ associated to a primitive aperiodic bijective substitution of length $\ell$ has a natural equicontinuous factor, namely the adding machine $(\Z_\ell,(+1))$, and only the orbit of $0\in\Z_\ell$ has singular fibres. We use the hierarchical symmetry defined by the substitution $\theta$ to compute the matrix form of 
$\M^{fib}_{0}(X_\theta)$ in  Theorem~\ref{thm-RMG},
  \[  E_0^{fib}(X_\theta)\backslash \{\Id\} = \M^{fib}_{0}(X_\theta)\cong
   M[G_\theta;I_\theta,\Lambda;A].\]    
$M[G_\theta;I_\theta,\Lambda;A]$ is a finite semigroup, to which we refer also as the {\em structural semigroup} of the substitution. The structure group $G_\theta$ has already appeared in work by  Lemanczyk and Mentzen in \cite{L-M} who identify it as the object whose centraliser completely encodes the {\em essential centraliser} of $(X_\theta,\sigma)$.

Provided that the smallest normal subgroup of $G_\theta$ which contains the group generated by the entries of $A$, which we denote  by $\Glstr$, is all of $G_\theta$, Theorem~\ref{thm-main2} gives a complete description  of $E(X_\theta)$ from  $E_0^{fib}(X_\theta)$. In particular,   $E(X_\theta)\backslash \Z$ is completely simple and there is a semigroup isomorphism
\begin{equation}\label{eq:one}E(X_\theta)\backslash \Z = \M(X_\theta)\cong M[\Gstr^{\Z_\ell/\Z}\rtimes \Z_\ell;I_\theta,\Lambda;A].\end{equation}
On the way to achieving this we also show that  $\Ef(X_\theta)$ is topologically isomorphic to 
$$\Ef(X_\theta) \cong (      M[G_\theta;I_\theta,\Lambda;A]            \cup \{\Id\}) \:\:\times \prod_{\stackrel{[z]\in\Z_\ell /\Z}{\scriptscriptstyle{[z]\neq [0]}}}\Gstr ,$$
and this isomorphism makes clear where the non-tameness comes from.

In general, $\Glstr$ can be a proper subgroup of $\Gstr$, but the quotient group $\Gstr/\Glstr$ is always a cyclic group. We call its order $h$ the  {\em generalised height} of the substitution. $h$ is at least as large as the classical height of a constant length substitution,  and we give in Section \ref{Examples} examples where it is  strictly larger. It is related to the topological spectrum of the dynamical system which is given by the action of $\Z$ on a minimal left ideal of $E(X_\theta)$, and $E(X_\theta)$ factors onto $\Z/h\Z$. In other words, $E(X_\theta)$ is a graded semigroup and its calculation can be reduced to its elements of degree $0$. In the case of nontrivial generalised height our result is Theorem~\ref{thm-main4}.  Here, with the assumption that the generalised height equals the classical height, we are able to describe $E(X_\theta)$ algebraically in a similar way as in the trivial height case, but with the structure group $G_\theta$ replaced with $\bar\Gamma_\theta$. However, when the generalised height is strictly larger than the classical height, the extension problem (\ref{eqn:exact-1}) remains unsolved.

In Section \ref{Ellis-group}, we apply our machinery to partly answer a  recent question of Auslander and Glasner in \cite{AG-2019}. They
define the notion of a {\em semi-regular} dynamical system, and ask  whether a minimal, point distal shift which is not distal can be semi-regular. They show that the  Thue-Morse  shift  is semi-regular. We extend their result, by showing  in Corollary \ref{Vag=Aut}, that 
the shift generated by a primitive aperiodic bijective substitution is  semi-regular. Implicitly, we relate the structure group $G_\theta$ to the {\em virtual automorphism group} that Auslander and Glasner define.

Our work is related to recent work of Staynova \cite{Staynova}, in which she computes the minimal idempotents of the Ellis semigroup for dynamical systems of bijective substitutions $\theta$ that are an {\em AI extension} of their maximal equicontinuous factor. In other words,  $(X_\theta, \sigma)$ is an isometric extension, via $f:X_\theta\rightarrow X_\phi$, of a constant length substitution shift $(X_\phi, \sigma)$, which is in turn an almost one-to-one extension, via $\pmax : X_\phi\rightarrow \Xmax$, of its maximal equicontinuous factor. Martin \cite{martin} characterises the bijective substitutions that are AI extensions of their maximal equicontinuous factor using a combinatorial property on the set of two-letter words allowed for $\theta$, namely that they are {\em partitioned} into sets according to what indices they appear at, as we scan all fixed points.  Staynova uses the functoriality of the Ellis semigroup construction, namely that a map between dynamical systems induces a semigroup morphism between their Ellis semigroups, 
and the fact that the Ellis semigroup of an equicontinuous system is a group, thus having exactly one idempotent. Using Martin's combinatorial condition, she first computes the preimages of that idempotent in $E(X_\phi)$. Apart from the identity map, all pre-images are minimal idempotents and live in the unique minimal left ideal. She then pulls this information up through the factor map $f$ to find that each of these minimal idempotents has two preimages, one for each minimal left ideal in $E(X_\theta)$. 

Our work goes beyond the results of Staynova in several respects. First, our techniques apply to all bijective substitutions. Indeed it is easy to define substitutions that do not satisfy Martin's criterion, so that their dynamical systems are not AI extensions of their maximal equicontinuous factor (see Example \ref{Martin}). Second, we do not only determine the idempotents, but the complete algebraic structure of  $E(X_\theta)$, at least if generalised height is not larger than classical height.
\bigskip

This paper is organised as follows. In Section \ref{preliminaries} we provide the necessary background on semigroups and the Ellis semigroup of a dynamical system, and study $\Z$-actions for 
which forward and backward proximality implies forward and backward asymptoticity, respectively.
In Section \ref{fibre-preserving} we study the Ellis semigroup for dynamical systems which have a single orbit of singular fibres under an equicontinuous factor map. In Section \ref{bijective}, we study in detail the Ellis semigroup of a bijective substitution dynamical system, and give an algorithm that computes its structural semigroup. In Section \ref{Ellis-group} we apply our results to investigate the virtual automorphism group of bijective substitution shifts. We end in Section \ref{Examples} with some examples.


\section{Preliminaries}\label{preliminaries}
The literature on the algebraic aspects of semigroups is vast and, although our work is partly  based on now classical results from the the forties we provide some background to the reader, {who may not be familiar with the basic material.} This can all be found in \cite{howie1995fundamentals}. We then recall the basic definitions and results on the Ellis semigroup of topological dynamical systems. These can mostly be found in \cite{Auslander} or \cite{hindman}.

\subsection{Semigroups, basic algebraic notions}

A {\em semigroup} is a set $S$ with an associative binary operation, which we denote multiplicatively. {Some of the semigroups in this paper  have an identity element, but some do not.  However they  will never have a $0$ element.}

A {\em normal inverse} to $s\in S$ is an element $t\in S$ such that $sts = s$, $tst = t$ and $st=ts$. A general element in a general semigroup need not  admit a normal inverse, but if it exists, it is unique. We may therefore denote it by $s^{-1}$.
A semigroup is called {\em completely regular} if every element admits a normal inverse. Completely regular semigroups have been studied in great detail \cite{petrich1999completely}. They are exactly the semigroups which may be written as disjoint unions of groups, i.e.\ 
$S=\bigsqcup_{i} \Gg_i$ such that multiplication restricted to $\Gg_i$ defines a group structure \cite[Theorem~II.1.4]{petrich1999completely}. The normal inverse of $s\in \Gg_i$ is then its group inverse in $\Gg_i$.  

Of particular importance in the analysis of a semigroup are its
idempotents and its ideals.
An idempotent of a semigroup $S$ is an element $p\in S$ satisfying $pp=p$.
The set of idempotents of $S$ is partially ordered via
$p\leq q$ if $p = pq = qp$. 
An idempotent is called {\em minimal} if it is minimal w.r.t.\ the above order. 
In general, we cannot expect to have minimal idempotents.

A (left, right, or bilateral) ideal of a semigroup $S$ is a {nonempty} subset $I\subseteq S$ satisfying
$SI\subseteq I$, $IS\subseteq I$, {or} $SI\cup IS\subseteq I$ {respectively}. The different kind of ideals will play different roles below. When we simply say ideal we always mean bilateral ideal. 
A semigroup is called {\em simple} if it does not have any proper ideal, and left simple if it does not have any proper left ideal. Note that a left simple semigroup is simple.
(Left, right, or bilateral)  
ideals are ordered by inclusion. A {\em minimal} (left, right, or bilateral) ideal 
is a minimal element w.r.t.\ this order, that is, a (left, right, or bilateral) ideal is minimal  if it does not properly contain another (left, right, or bilateral) ideal. In general, we cannot expect to have minimal ideals, but their existence in our specific context will be guaranteed for by compactness, see below. 
 
Whereas the intersection of two left ideals may be empty, this is not the case for the intersection of two bilateral ideals, or the intersection of a left ideal with a bilateral ideal.
Therefore the intersection of all bilateral ideals of a semigroup $S$ is either the unique minimal ideal of $S$, also called the {\em kernel} of $S$, or the intersection is empty, in which case $S$ does not admit a minimal ideal.  
The kernel of a semigroup without zero element is always simple \cite{howie1995fundamentals}.

Related to left and right ideals are the so-called Green's equivalence relations. Two elements $x,y\in S$ are $\mathcal L$-related if they generate the same left ideal, that is, there are $s,s'\in S$ such that $x=sy$ and $y = s'x$. Likewise  $x,y\in S$ are $\mathcal R$-related if they generate the same right ideal.
The intersection of the $\mathcal L$-relation with the $\mathcal R$-relation is called the $\mathcal H$-relation. The relation generated by  the $\mathcal L$-relation and the $\mathcal R$-relation, that is the join of $\mathcal L$ and $\mathcal R$,  is called the $\mathcal D$-relation. The relations $\mathcal L$ and $\mathcal R$ commute, so  $x$ and $y$ are $\mathcal D$-related if there is a $z$ such that $x$ and $z$ are $\mathcal L$-related and $z$ and $y$ are $\mathcal R$-related. Two results are of importance for what follows: First, an $\mathcal H$-class of $S$ which contains an idempotent is a subgroup of $S$ whose neutral element is the idempotent \cite[Corollary~2.26]{howie1995fundamentals}, and second, two $\mathcal H$-classes containing idempotents and which belong to a common $\mathcal D$-class must be isomorphic as groups \cite[Proposition~2.3.6]{howie1995fundamentals}.


\subsection{Simple semigroups and the Rees matrix form}\label{matrix-form}
Let $G$ be a group, let  $I$ and $\Lambda$ be  non-empty sets, and  let   
$A = (a_{\lambda i})_{\lambda\in \Lambda,i\in I}$ be a  $\Lambda\times I$ matrix with entries from $G$. Then the  {\em matrix semigroup $M[G;I,\Lambda;A]$} 
 is the set $I\times G \times  \Lambda$ together with the multiplication
\begin{equation*} (i,g,\lambda)(j,h,\mu) = (i, g a_{\lambda  j} h,\mu).\end{equation*}
The matrix $A$ is called the {\em sandwich matrix} and the group $G$ is called the {\em structure group}. 

It is an easy exercise to determine the idempotents and the left and the right ideals of $M[G;I,\Lambda;A]$. Indeed, an idempotent is of the form 
\begin{equation*}\label{idempotent-form}(i,a^{-1}_{\lambda i},\lambda),\end{equation*}
the  left ideals are the sets $I\times G\times\Lambda'$, $\Lambda'\subset \Lambda$, and the right ideals are $I'\times G\times \Lambda$, $I'\subset  I$. In particular, a completely simple semigroup has minimal left and minimal right ideals, namely those for which $\Lambda'$ or $I'$ contain a single element. These minimal left and right ideals are also the $\mathcal L$ and the $\mathcal R$ classes, and so the $\mathcal H$-classes are of the form
$\{i\}\times G\times\{\lambda\}$. $\{i\}\times G\times\{\lambda\}$
 is a subsemigroup of $M[G;I,\Lambda;A]$ which is a group. 
The identity element of this group is  the idempotent $(i,a^{-1}_{\lambda i},\lambda)$. It is isomorphic to $G$ via the isomorphism $(i,g,\lambda)\mapsto a_{\lambda i}g$.  
The normal inverse of 
$(i,g,\lambda)$ is $(i,a_{\lambda i}^{-1}g^{-1} a_{\lambda i}^{-1},\lambda)$.
In particular, a matrix  semigroup  as defined above is completely regular.

A {\em completely simple} semigroup is a simple semigroup which has minimal 
idempotents.
We have the following characterisation of completely simple semigroups.\footnote{Recall that we excluded the case that $S$ has a $0$-element. 
For semigroups with $0$-element there is an analogous but slightly different characterisation \cite{howie1995fundamentals}.}

{\begin{thm}[Rees-Suskevitch]\label{Rees-theorem}
A semigroup is completely simple if and only if it is isomorphic to a matrix semigroup $M[G;I,\Lambda;A]$ for some group $G$.
\end{thm}}
A proof of this theorem can be found in almost any textbook on semigroups. Since this result will be important in what follows we give a partial sketch of how to construct a Rees matrix from for a completely simple semigroup $S$. Proofs can be found in \cite{howie1995fundamentals}.
$S$ can be partitioned into its $\mathcal R$-classes, which we index by a set $I$. It can also be partitioned into its $\mathcal L$-classes, which we index by $\Lambda$. These partitions intersect yielding a partition into $\mathcal H$-classes. It can be shown that if $S$ is simple and contains an idempotent, then it consists of a single $\mathcal D$-class and that all its $\mathcal H$-classes contain an idempotent. In particular, $S$ is a disjoint union of groups which are all isomorphic. Moreover, each $\mathcal R$-class is a minimal right ideal and each $\mathcal L$-class is a minimal left ideal so that each $\mathcal H$-class is the intersection of a minimal right with a minimal left ideal.  Up to here, everything is canonical. But now we choose a minimal right ideal $R_{i_0}$ and a minimal left ideal $L_{\lambda_0}$ and set 
$$G := H_{i_0\lambda_0}$$
where we use the notation $H_{i\lambda} =  R_{i}\cap L_{\lambda}$.
As mentioned above, all other $\mathcal H$-classes are isomorphic to $G$, and indeed, given any $r_i \in H_{i\lambda_0}$ and $q_\lambda \in H_{i_0\lambda}$ 
\begin{eqnarray}\label{eq-iso1} 
&H_{i_0\lambda_0} \ni x \mapsto r_i x \in H_{i\lambda_0}& \\
\label{eq-iso2}
&H_{i_0\lambda_0} \ni x \mapsto  xq_\lambda \in H_{i_0\lambda}& 
\end{eqnarray}
are bijections which are group isomorphisms if $r_i$ and $q_\lambda$ are idempotents. Now the isomorphism between $G$ and the other $H_{i\lambda}$ will follow from the fact that $\mathcal L$ commutes with $\mathcal R$.
Taking into account these choices define the matrix $A=(a_{\lambda i})$
through
$$ a_{\lambda i} = q_\lambda r_i.$$
Then a direct calculation shows that
$$ M(G;I,\Lambda;A)\ni (i,g,\lambda) \mapsto r_i g q_\lambda \in S$$
yields the desired isomorphism.

We must ask how the Rees matrix form of a completely simple semigroup depends on the choices. The first choice is that of the right and left ideals indexed $i_0$ and $\lambda_0$, it defines the structure group $G=\mathcal H_{i_0 \lambda_0}$. 
A different choice will lead to a different but isomorphic structure group. An isomorphism can always be constructed using (\ref{eq-iso1},\,\ref{eq-iso2}). The second choice is that of the elements $r_i$ and $q_\lambda$. It affects the sandwich matrix. Indeed, one has the freedom to multiply any row of $A$ from the left and, independently, any column of $A$ from the right by an element of $G$ to obtain a sandwich matrix which defines an isomorphic semigroup. It is therefore possible to normalise $A$ in such a way that one of its rows and one of its columns contains only the identity element of $G$. More precisely, having chosen  
the right and left ideals indexed by $i_0$ and $\lambda_0$ we can always bring $A$ into its so-called {\em normalised} form by taking $r_i$ to be the unique idempotent of $\mathcal H_{i\lambda_0}$ and $q_\lambda$ to be the unique idempotent of $\mathcal H_{i_0\lambda}$ \cite[Theorem~3.4.2]{howie1995fundamentals}. Up to the choice of $i_0$ and $\lambda_0$ this {\em normalised} Rees matrix form is then unique. Since any pair $(i,\lambda)\in I\times \Lambda$ determines a unique idempotent of $S$ we can also formulate this as follows: once we have chosen an idempotent of $S$, typically denoted $e$, we obtain a unique normalised Rees matrix form for $S$. To be precise we call this {\em the normalised Rees matrix form for $S$ w.r.t.\ $e$.} In what follows the use of $e$ refers to this chosen minimal idempotent.

Given a normalised matrix semigroup $M(G;I,\Lambda;A)$ w.r.t.\ $e$,
we call the subgroup $\Gamma$ of $G$ generated by the coefficients 
$a_{\lambda i}$ of $A$ the {\em little structure group}.  
\begin{lem} \label{lem-lsgroup}
Consider a normalised matrix semigroup $M(G;I,\Lambda;A)$ w.r.t.\ $e=(i_0,1,\lambda_0)$. 
The subsemigroup of $M(G;I,\Lambda;A)$ which is generated by the idempotents is equal to $M(\Gamma;I,\Lambda;A)$. 
\end{lem}
\begin{proof} Let $K$ be the subsemigroup of $M(G;I,\Lambda;A)$ which is generated by the idempotents. By definition of the little structure group, $(i,G,\lambda)\cap K \subset (i,\Gamma,\lambda)$.
Normalisation implies $a_{\lambda i}=1$ provided $i=i_0$ or $\lambda=\lambda_0$. Given $a_{\lambda i}$ we know that $(i,a_{\lambda i}^{-1},\lambda)$ is an idempotent. Hence
$$(i_0,a_{\lambda i},\lambda_0)=(i_0,1,\lambda)(i,a_{\lambda i}^{-1},\lambda)(i,1,\lambda_0)\in (i_0,G,\lambda_0)\cap K.$$
This shows that $(i_0,\Gamma,\lambda_0)\subset (i_0,G,\lambda_0)\cap K$.  Hence also
$$(i,\Gamma,\lambda) = (i,1,\lambda_0) (i_0,\Gamma,\lambda_0)(i_0,1,\lambda) 
\subset (i,G,\lambda)\cap K.$$
This shows that $M(G;I,\Lambda;A)\cap K = M(\Gamma;I,\Lambda;A)$.
\end{proof}

\subsubsection{Example }\label{ex:matrix-semigroup-ex}
We consider a class   
of matrix semigroups $M[G;I,\{\pm\};A]$ which will play a major role later. For this family,  $G$ is a finite group with neutral element  $\one$
and $I\subseteq G$ is a subset which generates $G$. Fix $g_0\in I$. Let $\Lambda = \{+,-\}$ be a set of two elements. Define the $\Lambda\times I$ matrix $A=(a_{\lambda i})_{\lambda i}$
\begin{equation}\label{eq:matrix} a_{+\,g} = \one \qquad a_{-\,g} = g_0 g^{-1} \end{equation}
Then $   M[G;I,\{\pm\};A]    $ has   $2|I||G|$ elements 
of which $2|I|$ are idempotents.
Note that $M[G;I,\{\pm\};A]$ is normalised w.r.t.\ the idempotent $e=(g_0,\one,+)$.
\begin{lem}\label{lem-JSG} With the notation above, the little structure group of 
$M[G;I,\{\pm\};A]$ is the group generated by $g h^{-1}$, $g,h\in I$.
\end{lem}
\begin{proof} This follows directly from $g h^{-1} = a_{-g}^{-1} a_{-h}$.
\end{proof}

\subsection{Compact semigroups}\label{compact semigroups}
A {\em topological} semigroup is a semigroup $S$ equipped with a topology in which the multiplication map $S\times S\to S$ is (jointly) continuous. A semigroup (equipped with a topology) is called {\em right-topological} if, for any $s\in S$ {\em right multiplication} $\rho_s: S\to S$, $\rho_s(t) := ts$ is continuous. Note that this is equivalent to  multiplication $S\times S\to S$ being continuous in the {\em left} variable which is why 
the term left-topological is also sometimes employed. We follow here the terminology of \cite{hindman}. A topological semigroup is right-topological and left-topological (with the obvious definition), but the converse need not be true. 

Let $X$ be a topological space. The set $F(X)$ of functions $X\to X$ with the topology of pointwise convergence is 
the same as the infinite Cartesian product $X^X$ with product topology. It is  perhaps 
the simplest example of a right-topological semigroup, the semigroup product being composition of functions. If $X$ is compact then $F(X)$ is compact. Only if $X$ is discrete is $F(X)$ a topological semigroup. 
 
Let $\pi:X\to Y$ be a continuous surjection. We call the preimage $\pi^{-1}(y)$ the $\pi$-fibre of $y$. Let $F^{fib}(X)\subset F(X)$ be the subsemigroup of all functions $X\to X$ which preserve the $\pi$-fibres. Since fibres are closed subspaces of $X$, $F^{fib}(X)$ is a closed subsemigroup of $F(X)$. 
We can view $f\in F^{fib}(X)$ as a function $\tilde f$ on $Y$,
\begin{equation} \label{eq:definition-f-tilde}
\tilde f : y
\mapsto  f|_{\pi^{-1}(y)}\end{equation}
which, evaluated at $y$ is the restriction of $f$ to $\pi^{-1}(y)$, $\tilde f(y)(x) = f(x)$ for $x\in \pi^{-1}(y)$. This identification $f \mapsto \tilde f$
 yields a {\em topological isomorphism}, i.e. a homeomorphism which is also a semigroup isomorphism,   between 
$F^{fib}(X)$ and the direct product $\prod_{y\in Y} F(\pi^{-1}(y))$ 
where the semigroup multiplication in the latter space is $\tilde f_1\tilde f_2(y) = \tilde f_1(y)\circ \tilde f_2(y)$ and 
we equip it with the product topology, $F(\pi^{-1}(y))$ still carrying the topology of pointwise convergence. Recall that $F(\pi^{-1}(y))$ is a topological semigroup if the fibre of $y$ is finite. By definition of the product topology we therefore get that $\prod_{y\in Y} F(\pi^{-1}(y))$ is a topological semigroup provided all fibres are finite.

For compact semigroups one has the following results concerning their kernels and corresponding Rees matrix form.  
\begin{thm}\label{thm-Rees-str-1}
Let $S$ be a compact right-topological semigroup. Then $S$ admits a 
kernel $\M(S)$ which contains all minimal idempotents, so  that $\M(S)$ is isomorphic to a matrix semigroup. Furthermore, all minimal left ideals are compact and homeomorphic, and two $\mathcal H$-classes of $\M(S)$ which belong to the same minimal right ideal are topologically isomorphic.  
\end{thm}

This theorem is discussed in \cite[Corollary~2.6 and Theorem~2.11]{hindman}. 
The essential input from the compact topology is the existence of an idempotent and the continuity of the map (\ref{eq-iso2}). 
We mention that in general, minimal right ideals are not closed, nor are $\mathcal H$-classes closed, nor are two $\mathcal H$-classes topologically isomorphic which do not belong to the same minimal right ideal. $\M(S)$ is then not topologically isomorphic to a matrix semigroup. 

One consequence of  Theorem \ref{thm-Rees-str-1} will be particularly important below, namely that for any minimal idempotent $p$ of a compact right-topological semigroup $S$, $pSp$ is a group. Indeed, the chain of inclusions
\begin{equation} \label{pEp-group} p\M(S) p\subset pSp = pSp p \subset p\M(S) p. \end{equation} 
shows that $pSp$ is isomorphic to the structure group of the kernel $\M(S)$.

If the multiplication of $S$ is jointly continuous then the topological aspects of  
Theorem \ref{thm-Rees-str-1} can be strengthened. 
We can equip the normalised Rees matrix form $M[G;I,\Lambda;A]$ (w.r.t.\ $e=(i_0,1,\lambda_0)$) of $\M(S)$ with the following topology: We identify $G$ with $H_{i_0\lambda_0}=eSe$, $I$ with the set of idempotents of $L_{\lambda_0}$, $\Lambda$ with the set of idempotents of $R_{i_0}$, and we equip all these subsets of $S$ with the relative topology, and finally $I\times G\times \Lambda$ with the product topology. 
Under the assumption that $S$ is a compact topological semigroup, we have that $G$ is a compact topological group, $I$ and $\Lambda$ compact subsets and the semigroup product on $M[G;I,\Lambda;A]$ is jointly continuous. 
\begin{thm}\label{thm-Rees-str-2}
Let $S$ be a compact topological semigroup. Then  $\M(S)$ is topologically isomorphic to the normalised matrix semigroup $M[G;I,\Lambda;A]$.
\end{thm}
A proof of this theorem can be found in \cite[Theorem~3.21]{carruth1983theory}. 
As now also the map (\ref{eq-iso1}) is continuous, all $\mathcal H$-classes of $M(S)$ are closed and topologically isomorphic.

\subsection{Extensions of groups by completely simple semigroups}
We now use the above description of completely simple semigroups to study extensions
$$ K\hookrightarrow S \stackrel{\pi} \twoheadrightarrow Y $$
 where $Y$ is a group with neutral element $y_0$, 
$S$ a semigroup, $\pi$ a semigroup epimorphism and $K$ the kernel of $\pi$,
$$K = \{s\in S: \pi(s)=y_0\}.$$ 
$K$ is a subsemigroup of $S$ which is closed if $S$ and $Y$ are right-topological and $\pi$ continuous.

If $e\in S$ is an idempotent then $\pi(e)$ must be an idempotent, hence equal to $y_0$ so that we obtain a restricted extension
$ eKe\hookrightarrow eSe \stackrel{\pi_e} \twoheadrightarrow Y $ where $\pi_e$ is the restriction of $\pi$ to $eSe$. 
If moreover $e$ is an idempotent in the kernel of $S$ then by  \eqref{pEp-group},$eSe$ is a group, as is  $eKe$, 
so that 
the restricted extension is an extension of groups.

A semigroup $S$ is {\em regular} if for any $s\in S$ there exists $t\in S$ such that $s=sts$.
Clearly, any completely regular semigroup is regular.
\begin{prop}\label{extension-semigroup}
Consider an extension $ K\hookrightarrow S \stackrel{\pi} \twoheadrightarrow Y $ of a group $Y$ by a completely simple semigroup $K$, where $S$ is regular. Then $S$ is completely simple. If $K$ has normalised Rees matrix form $M[G;I,\Lambda;A]$ w.r.t.\ an idempotent $e$ then $S$ has normalised Rees matrix form $M[\mathcal G;I,\Lambda;A]$ 
w.r.t.\ $e$, where $\mathcal G=eSe$ is the extension of $Y$ by $G=eKe$
determined by the exact sequence of groups $eKe \hookrightarrow eSe \stackrel{\pi_e}\twoheadrightarrow Y$.
\end{prop}
\begin{proof} We first show that $S$ must be completely simple. Let $M\subset S$ be an ideal. Then $M\cap K$ is an ideal of $K$. As $K$ is simple, $M\cap K = K$. Thus $M$ contains all idempotents of $S$. Let $s\in S$ and $t\in S$ such that $s=sts$. Then $ts$ is an idempotent and so we see that $S\subset SK\subset SM\subset M$. Therefore $S$ is completely simple.

Let $M[\mathcal G;I,\Lambda;A]$ be the normalised Rees matrix form of $S$ w.r.t.\ $e$, in particular $\mathcal G=eSe$. Since $K$ contains the subsemigroup generated by the idempotents of $S$,
by Lemma~\ref{lem-lsgroup}, the coefficients of $A$ belong to $G:= eKe = \ker \pi_e$. Hence 
$M[G;I,\Lambda;A]$ is well-defined and, with $e=(i_0,1,\lambda_0)$ and  
$(i_0,1,\lambda_0)(i,g,\lambda)(i_0,1,\lambda_0)=(i_0, g ,\lambda_0)$
we obtain
$$\pi(i,g,\lambda) = \pi(i_0,g,\lambda_0) =  \pi_e(g)$$
so that $K= M[\ker\pi_e;I,\Lambda;A]$.
Since  the normalised Rees matrix form w.r.t.\ $e$ is unique we see that $I,\Lambda$ and $A$ are completely determined by $K$. 
\end{proof}

We will apply Proposition \ref{extension-semigroup} to a situation in which $Y$ is a topological group and $K$ is a topological semigroup, but $S$ is only a right-topological semigroup.  Therefore we cannot  conclude that $S$ is topologically isomorphic to $M[\mathcal G;I,\Lambda;A]$. Indeed, the group $\mathcal G=eSe$ is only closed if  the right ideal containing $e$ is closed, which we cannot expect. The interest in the above construction is therefore only algebraic. 
It is particularly useful if the group extension is split so that $\mathcal G$ is a semidirect product of $G$ with $Y$.

\subsection{Ellis semigroup of a dynamical system}
Given a dynamical system $(X,T,\sigma)$ the family of homeomorphisms
$\{\sigma^t | t\in T\}$ is a subsemigroup of $F(X)$.
Its closure, denoted $E(X,T, \sigma)$,  or simply $E(X)$ if the rest is understood, is still a semigroup, called the   {\em Ellis semigroup} (or {\em enveloping semigroup}) of the dynamical system. 
Since $X$ is compact the set of all functions $X\to X$ is compact in the topology of pointwise convergence and so $E(X,T, \sigma)$ is a compact \rst\ semigroup, by construction. 

The Ellis semigroup is closely related to the {\em proximality} relation. Given a metric $d$ on $X$ which generates the topology,  a pair of points $x,x'$ are  {\em proximal} if  $\inf_{t\in T} d(\sigma^t(x),\sigma^t(x')) = 0$. The proximal relation does not depend on the choice of metric (which generates the topology).  Its relation with the Ellis semigroup is the following:
\begin{thm}\cite[Chapter 3, Proposition 8]{Auslander} 
Let $E(X)$ be the Ellis semigroup of a dynamical system $(X,T, \sigma)$. Two points $x$ and $y$ are proximal if and only if there exists $f\in E(X)$ such that $f(x)=f(y)$.
\end{thm}
In particular we see that, given any idempotent $p\in E(X)$ and $x\in X$, the points $p(x)$ and $x$ are proximal.

\subsection{Complete regularity for $\Z$-actions}\label{Section:comp reg}
In this section we provide a criterion for complete regularity of the Ellis semigroup for $\Z$ actions. We will see below that it is satisfied by the dynamical systems defined by bijective substitutions. 

Since the union of the closure of two sets is the closure of their union we can decompose 
\begin{equation}\label{eq-fb}
E(X) = E(X,\Z^+) \cup E(X,\Z^-)
\end{equation}
where $E(X,\Z^\pm)$ is the closure of $\{\sigma^t | t\in \Z^\pm\}$. This allows us to compute the elements of $E(X)$ by looking independently, forward in ``time", and backwards in ``time". 

We say that two points $x,x'\in X$ are {\em forward proximal} if 
$$\inf_{t\in\Z^+ } d(\sigma^t(x),\sigma^t(x')) = 0$$ 
We say that two points $x,x'\in X$ are
{\em forward asymptotic} if $$\lim_{t\to +\infty} d(\sigma^t(x),\sigma^t(x')) = 0$$
Similarly, we define {\em backward} proximality and asymptoticity using $\sigma^{-1}$ in place of $\sigma$.
Clearly sequences which are forward asymptotic are forward proximal. The following lemma is related to the work of \cite{Blanchard} in which the adherence semigroup of a $\Z^+$-action is defined and analysed.
\begin{lem}\label{lem:prox=asym}
Let $(X,\sigma)$ be a dynamical system for which forward proximality agrees with forward asymptoticity. Then $E(X,\Z^+)$ has a unique minimal left ideal $\M(X,\Z^+)$ and contains besides this ideal only $\Z^+$.
\end{lem}

\begin{proof} 
An element $f\in E(X,\Z^+)\backslash \Z^+$ is the limit of a generalised sequence $(\sigma^{t_\nu})_{\nu}$, $t_\nu\in\Z^+$ which is not in $\Z^+$. Hence the generalised sequence $(t_\nu)_\nu$ has the property that for any finite $N\in\Z^+$ there exists $\nu_0$ such that 
$t_{\nu}\geq N$ for all $\nu >\nu_0$. In particular, if $x$ and $y$ are forward asymptotic points then $\lim_\nu d(\sigma^\nu(x),\sigma^\nu(y))=0$, and hence $f(x) = f(y)$. 

$E(X,\Z^+)$ is also a compact \rst\ semigroup and hence has minimal left ideals and minimal idempotents. Furthermore, $x,y\in X$ are forward proximal if and only if there exists $f\in  E(X,\Z^+)$ such that $f(x) = f(y)$. 
Let $p\in E(X,\Z^+)$ be any idempotent. For any $x\in X$, $p(x)$ is forward proximal to $x$, and by our assumption therefore forward asymptotic to $x$. This implies that if $f\in E(X,\Z^+)\backslash \Z^+$, then $f(p(x))=f(x)$. Since $x$ was arbitrary we find $f=fp$. 

This identity 
shows that any  $f\in E(X,\Z^+)\backslash \Z^+$ lies in the ideal generated by the idempotent $p$. If $p$ is minimal then this ideal is a minimal left ideal. Since $p$ can be any minimal idempotent there can only be one minimal left ideal.
\end{proof}
Note that the unique minimal left ideal $\M(X,\Z^+)$ of the previous lemma is the kernel of $E(X,\Z^+)$.
\begin{cor}\label{cor:prox=asym}
Let $(X,\sigma)$ be a dynamical system for which forward proximality agrees with forward asymptoticity and backward proximality agrees with backward asymptoticity. Then $E(X)$ is completely regular. If moreover the forward and the backward proximality relations are non-trivial (not diagonal) then $E(X)$ is the disjoint union of its kernel $\M(X)$ with the acting group $\Z$.\end{cor}
\begin{proof}
Minimal left ideals which contain idempotents are completely simple and hence, by the Rees structure theorem, disjoint unions of groups. 
Therefore Lemma \ref{lem:prox=asym}
implies that $E(X,\Z^+)\backslash \Z^+$ is completely regular. $E(X)$ is thus a union of completely regular sub-semigroups. Hence any element of $E(X)$ has an inverse with which it commutes.

To proof the second statement we first show that $\M(X,\Z^+)$ is a minimal left ideal in $E(X)$. 
Let $f\in E(X)$, $g\in E(X,\Z^+)\backslash \Z^+$. So $f = \lim \sigma^{n_\nu}$ and $g = \lim \sigma^{m_\mu}$, however with $m_\mu\to+\infty$. Then $fg = \lim_\nu \sigma^{n_\nu} g$.
Since $\sigma^{n_\nu} g = \lim_{\mu}  \sigma^{n_\nu+m_\mu} \in E(X,\Z^+)$ and $E(X,\Z^+)$ is closed 
we have $fg\in E(X,\Z^+)$. Suppose that $fg=\sigma^n$ for some $n\in \Z$. As $g\in \M(X,\Z^+)$ we have $gp = g$ for some idempotent $p\in \M(X,\Z^+)$. It follows that 
$1=\sigma^{-n}fg= \sigma^{-n}fgp = p$. This implies that $E(X,\Z^+)$ is a group and thus contradicts the assumption that the forward proximality relation is non-trivial \cite{Auslander}. 
Hence $fg\notin \Z$ so that by Lemma~\ref{lem:prox=asym},
$E(X)\M(X,\Z^+)\subset \M(X,\Z^+)$ and moreover $\M(X,\Z^+)\cap \Z=\emptyset$. 
To show minimality of $\M(X,\Z^+)$ it suffices to show that all idempotents of $\M(X,\Z^+)$ are minimal in $E(X)$. 
Let $q\in \M(X,\Z^+)$ and $p\in E(X)$ be idempotents such that $p\leq q$. This means that $pq = qp = p$. As we just showed, $p=pq\in \M(X,\Z^{+})$. 
But then $p=q$ as $q$ is minimal in $E(\M,\Z^+)$. 

As the kernel of a semigroup is the union of its minimal left ideals we now have shown that
$\M(X)=\M(X,\Z^+)\cup \M(X,\Z^-)$ and that $\M(X)\cap \Z=\emptyset$.
\end{proof}
The arguments in the second part of the proof were adapted from \cite{BargeKellendonk} where it is also shown that, for minimal $\Z$-actions on totally disconnected compact metric spaces, the condition that forward proximality agrees with forward asymptoticity and backward proximality agrees with backward asymptoticity is also necessary for complete regularity.


\subsection{Equicontinuous factors and the structure of $E(X)$}\label{semigroup-of-factor}

In this section $T$ is an abelian group. When we have a factor map between two dynamical systems, the acting group is the same.  A dynamical system $(X, \sigma)$ is called  {\em equicontinous} if the family of homeomorphisms 
$\{\sigma^t, t\in T\}$ is equicontinuous.  If the action is transitive then this is the case if and only if, for any choice of $x_0\in X$ there is an abelian group structure on $X$ (denoted additively) such that $x_0$ is the identity element and $\sigma^t(x) = x+ \sigma^t(x_0)-x_0$. This group structure is topological.

Moreover, for a minimal equicontinuous system and w.r.t.\ the above group structure on $X$,
$ev_{x_0} : E(X) \to X$ is an isomorphism of topological groups, where $ev_{x_0}$ is evaluation at the point $x_0\in X$, $ev_{x_0}(f) = f(x_0)$ \cite[Chap.~3, Theorem~6]{Auslander}.

An {\em equicontinuous} factor is a factor $\pi:(X, \sigma)\to (Y,\delta)$ such that $(Y,\delta)$ is equicontinuous. As with any factor map, $\pi$ induces a continuous semigroup morphism  
$\pi_*: E(X) \to E(Y)$ via $\pi_*(f)(y) = \pi(f(x))$ where $x$ is any pre-image of $y$ under $\pi$. As $(Y,\delta)$ is equicontinuous $ev_{y_0}:E(Y) \to Y$ is a semigroup isomorphism where $y_0$ is the identity element in $Y$. We denote by $\tilde \pi:E(X)\to Y$ the composition $ev_{y_0}\circ \pi_*$,  which is also a continuous surjective semigroup morphism.

\begin{definition}\label{def:Efib} 
Let $\pi:(X,\sigma)\to (Y,\delta)$  be an  equicontinuous factor.
Define $\Ef(X)$ to be the subsemigroup of  $E(X)$ which consists of those elements which preserve the $\pi$-fibres $\pi^{-1}(y)$, $y\in Y$. 
\end{definition}
In other words, $\Ef(X)$ is the kernel of the continuous semigroup morphism $\tilde\pi$ and therefore a closed subsemigroup. We summarize this situation with the exact sequence of right-topological semigroups
\begin{equation}\label{eq-ext}
 \Ef(X) \hookrightarrow E(X) \stackrel{\tilde\pi} \twoheadrightarrow Y 
\end{equation}
in which the involved maps are continuous semigroup morphisms.  While $E(X)$ is only right-topological, $Y$ is topological. As we will see below, under certain circumstances, $\Ef(X)$ is also a topological semigroup. 

\section{The fibre-preserving part $\Ef(X)$}\label{fibre-preserving}
In this section we investigate the fibre-preserving part $\Ef(X)$ of $E(X)$ for dynamical systems which factor onto an equicontinuous system, $\pi:X\to Y$. 
We call a point $y\in Y$ is {\em regular} (for $\pi$) 
if the proximal relation restricted to $\pi^{-1}(y)$ is trivial. 
Otherwise we call the point {\em singular} (for $\pi$). 

\begin{prop}
Suppose that $\pi:(X,\sigma)\to (Y,\delta)$ is an equicontinuous factor map whose fibres $\pi^{-1}(y)$ are all finite. Then $\Ef(X)$ is a compact topological semigroup. 
\end{prop}
\begin{proof}
$\Ef(X)$ is a compact subsemigroup of $F^{fib}(X)$,  defined in Section \ref{compact semigroups}. By  assumption, all $F(\pi^{-1}(y))$ are (trivially) topological semigroups. Therefore  the semigroup multiplication of $\prod_{y\in Y} F(\pi^{-1}(y))$ is jointly continuous.
As $\Ef(X)$ is a closed subsemigroup of $F^{fib}(X)$ its product is also jointly continuous.
\end{proof} 
We can now apply Theorem \ref{thm-Rees-str-2} to conclude the following.
\begin{cor}\label{cor-ff}
Suppose that $\pi:(X,\sigma)\to (Y,\delta)$ is an equicontinuous factor map whose fibres $\pi^{-1}(y)$ are all finite. Then the kernel $\M^{fib}(X)$ of $\Ef(X)$ is topologically isomorphic to its normalised Rees matrix form. 
\end{cor}

We now consider more closely the algebraic structure of $\Ef(X)$.
To simplify the notation we drop the reference to $X$ and denote it by $\Ef$.

When identifying $\Ef$ with a subsemigroup of 
$F^{fib}(X)\cong \prod_{y\in Y} F(\pi^{-1}(y))$ 
we observe that it belongs actually to the smaller semigroup 
$\prod_{y\in Y} \Ef_y$ where 
\begin{equation*}\label{eq:Efib restriction}
\Ef_y := \Ef|_{\pi^{-1}(y)},
\end{equation*}
the restriction of $\Ef$ to the fibre $\pi^{-1}(y)$. Indeed, any $f\in \Ef$ corresponds to a function $\tilde f$ on $Y$ whose value $\tilde f(y)$ belongs to $\Ef_y$.  $\Ef_y$ is a compact subsemigroup of $F(\pi^{-1}(y))$.
Moreover, since the elements of $E$ commute with the action $\sigma$ of $T$ the functions $\tilde f$ have to be {\em covariant} in the sense that 
\begin{equation*}\label{eq-cov}
\tilde f(\delta^t(y)) = \sigma^t \tilde f(y)\sigma^{-t}
\end{equation*}
for all $y\in Y$ and $t\in T$. In other words, $\Ef$ is a subsemigroup of 
\begin{equation*}\label{eq:definiton-of-Cf} \Cf  := \{\tilde f\in \prod_{y\in Y} \Ef_y
:  \tilde f \;\mbox{is covariant}\},\end{equation*}
again equipped with the pointwise semigroup multiplication 
$(\tilde f_1\tilde f_2)(y) = \tilde f_1(y)\circ \tilde f_2(y)$.  Equipped with the product topology,
$\Cf$ is compact and the inclusion $\Ef\subset \Cf$ 
is continuous. 

Recall {by \eqref{pEp-group}} that if $p$ is any minimal idempotent then  $pEp$ is a group.
We fix any such minimal idempotent $e$, recalling that  different minimal idempotents define isomorphic groups.
As $\tilde\pi: E \to Y$ is onto, and $\tilde \pi (efe)= \tilde \pi (f)$, the restriction 
$\tilde\pi: eEe \to Y$ is also onto. 
A {\em lift under $\tilde\pi$} is a right inverse $s:Y\to eE e$ to $\tilde\pi:e E e\to Y$, i.e.\ it satisfies $\tilde \pi \circ s = \Id$. A lift always exists  by the axiom of choice.
We do not demand that it is continuous, nor, for the time being, that it preserves the group structure. But we can and do demand that it satisfies
$s(\delta^t(y)) = \sigma^t s(y)$ for all $t\in T$, and also that $s(y)^{-1}=s(-y)$ for each $y\in Y$. We impose the latter condition now although we will not use it until
Proposition \ref{prop-shomo}.

Given a lift $s:Y\to e E e$, we define $\Phi_{y_1}^{y_2}:\Ef\to \Ef$ by
\begin{equation}\label{eq-conj}
\Phi_{y_1}^{y_2} (f)  =  s(y_2-y_1) f \, s(y_2-y_1)^{-1} 
\end{equation}
where $s(z)^{-1}$ is the group inverse to $s(z)$. Although we do not include this in our notation, it must be kept in mind that  $\Phi_{y_1}^{y_2}$ depends on the choice of lift. 
Since $s(\delta^t(y)) = \sigma^t s(y)$ we have 
$\Phi_{y_1}^{\delta^t(y_2)} (f) = \sigma^t \Phi_{y_1}^{y_2} (f) \sigma^{-t}$.

Note that $\Phi_{y_1}^{y_2}$ also defines  a map from $\Ef_{y_1}$ to $\Ef_{y_2} $, namely if $\varphi\in \Ef_{y_1} $ and $f$ is an element of $\Ef$ which restricts to $\varphi$ on $\pi^{-1}(y_1)$, that is $\varphi = \tilde f(y_1)$ in the notation above,
 then $\Phi_{y_1}^{y_2} ( \varphi )$ is defined to be the restriction of 
$\Phi_{y_1}^{y_2} (f) $ to $\pi^{-1}(y_2)$. This does not depend on the choice of $f$, as $s(y_2-y_1)^{-1}$ maps $\pi^{-1}(y_2)$ to $\pi^{-1}(y_1)$.
\begin{lem}\label{structure}
 Let $e\in E(X)$ be a minimal idempotent.
\begin{enumerate}
\item $e\Ef_y e$ is a group.
\item If $y$ is regular then $e\Ef_y e = \Ef_y$. 
\item For any $y_1, y_2$, the restriction $\Phi_{y_1}^{y_2}:e\Ef_{y_1} e\to e\Ef_{y_2} e$
is a group isomorphism.
\end{enumerate}
\end{lem}

\begin{proof}
$e\Ef_y e$ is entirely determined by the action of $e\Ef e$ on $e\pi^{-1}(y)$. It is hence the homomorphic image of a group.   

Idempotents must act like the identity on a regular fibre, as can be seen as follows:
The points $e(x)$ and $x$ are proximal. In a regular fibre this can only be the case if $e(x)=x$. Hence $e\Ef_y e = \Ef_y$ if $y$ is regular.

Let $s:Y\to e E e$ be a right inverse to $\tilde\pi:eE e\to Y$; 
$s(z)$ restricts to a map  $e\pi^{-1}(y)\to e\pi^{-1}(y+z)$ whose inverse is the restriction of $s(z)^{-1}$, as $s(z)^{-1}s(z) = s(z)s(z)^{-1} = e$. Hence 
$\Phi_{y_1}^{y_2}$ is conjugation with a bijection.
\end{proof}
\begin{definition}\label{def:structure group} We call the group determined up to isomorphism by Lemma~\ref{structure} 
the {\em structure group} of the factor system $(X,\sigma)\stackrel{\pi}\to (Y,\delta)$ and denote it by $G_\pi$.  
\end{definition}

\begin{definition} {We say that the system $(X,\sigma)$ is a {\em unique singular orbit system} if it admits an equicontinuous factor which has a single orbit of singular points.}
\end{definition}
We now specialize to the context of unique singular orbit systems and fix a singular point $y_0\in Y$. Define $\Et\subset \Ef$ to be
\begin{equation}\label{eq:def T}
\Et=\{ f\in \Ef: f(x)= x \mbox{ for all $x$ in a regular fibre}\}.
\end{equation}
Since idempotents can only project proximal points, and regular fibres contain no proximal pairs, so idempotents belong to $\Et$.
 However $\Et$  may be larger. 
Given the minimal idempotent $e$, $e\Et e$ is a subsemigroup of $e\Ef e$. We claim that it is even a normal subgroup. Indeed, if $f\in e\Et e$ then  its inverse in $e\Ef e$ also acts trivially on regular fibres and so belongs to $e\Et e$. Furthermore, an element $g\in e\Ef e$ acts bijectively on regular fibres and hence $gfg^{-1}$ acts as $g g^{-1}=\Id$ on them. 

Let $\Et_{y_0}$ be the restriction of $\Et $ to $\pi^{-1}(y_0)$; it is a subsemigroup of $\Ef_{y_0}$. 
Then $e\Et_{y_0} e$ is the restriction of $e\Et e$ to $e\pi^{-1}(y_0)$; it is a normal subgroup of $e\Ef_{y_0} e$. We now use the maps $\Phi_{y_0}^{y}$ from (\ref{eq-conj}) to transport the group $e\Et_{y_0} e$ along $Y$ and define the subsemigroup of $\Cf$
\begin{equation}  \label{eq:def Ct} 
\Ct := 
\{ \tilde f\in \Cf :    \tilde f (y_0) \in \Et_{y_0}, \mbox{ and } \tilde f (y) \in \Phi_{y_0}^{y}(e\Et_{y_0} e) \mbox{ for all $y$ regular} \}.
\end{equation}
Although the map $\Phi_{y_0}^{y}$ depends on the choice of a lift $s:Y\to eEe$ for $\tilde \pi$, the space $\Ct$ does not. 
Indeed, if we take another lift to obtain a map ${\Phi'}_{y_0}^{y}$ then $\Phi_{y_0}^{y}(\tilde f)(y_0)$ will differ from ${\Phi'}_{y_0}^{y}(\tilde f)(y_0)$ by a conjugation with an element $h\in e \Ef_{y_0} e$, which does not matter as $e\Et_{y_0} e$ is a normal subgroup of 
$e\Ef_{y_0} e$. 
By covariance $\Ct$ does not depend on the choice of $y_0$ in the unique orbit of singular points. 
Finally, the dependence of $\Ct$ on the choice of minimal idempotent $e$ can be controlled with the isomorphisms (\ref{eq-iso1},\ref{eq-iso2}). If the singular fibre $\pi^{-1}(y_0)$ is finite then the isomorphisms are bicontinuous by Theorem ~\ref{thm-Rees-str-2} and Corollary ~\ref{cor-ff}.


\begin{thm} \label{thm-main} 
Let $(X,\sigma)$ be a minimal unique singular orbit system.  $\Ct$ is a subsemigroup of $\Ef$.
\end{thm}
\begin{proof}
Let $g \in \Et$ and $y\in Y$. Then  $f:= \Phi_{y_0}^{y} (g)$ belongs to 
$\Ct\cap \Ef$. Indeed,  $\Phi_{y_0}^{y} (g)=\Phi_{y_0}^{y} (ege)\in \Phi_{y_0}^{y} (e\Et_{y_0}e)$.

By definition $g$ acts non-trivially only on the fibres of the $T$-orbit of $y_0$. Hence
$f:= \Phi_{y_0}^{y} (g)$ acts non-trivially only on the fibres of the $T$-orbit of $y$. 
As $(\Phi_{y_0}^{y})^{-1}(\tilde f(y))$ is the restriction of  $ege$ to $e\pi^{-1}(y_0)$
we find that, given any regular point $y$ and any $g\in e\Et_{y_0} e$, $\Ct\cap \Ef$ contains the function $f$ which satisfies $(\Phi_{y_0}^{y})^{-1}(\tilde f(y)) = g$ and $\tilde f(y') = \Id$ for any point $y'$ in another orbit. By taking finite products of such functions we see that
$\Ct\cap \Ef$ contains, for any choice of  $k$ points $y_1,\cdots,y_k$ in distinct regular orbits 
and any choice of $k+1$ elements $g_i\in \Et_{y_0}$, $i=0,\cdots,k$ a function $f$ such that 
$(\Phi_{y_0}^{y_i})^{-1}(\tilde f(y_i)) = eg_ie$, $i\geq 1$, $\tilde f(y_0) = g_0$, and $\tilde f(y') = \Id$ for a point $y'$ in another orbit.
By definition of the topology of pointwise convergence and since covariance is a closed relation, the set of these elements is dense in $\Ct$. Since $\Ef$ is the kernel of a continuous map it is closed; it hence contains $\Ct$. 
\end{proof}
\begin{cor} \label{cor-main} 
Let $(X,\sigma)$ be a minimal unique singular orbit system. If $\Ef_{y_0}=\Et_{y_0}$ then $\Ef$ is  topologically isomorphic to $\Cf$. 
\end{cor}
\begin{proof}
If $\Ef_{y_0}=\Et_{y_0}$ then $\Phi_{y_0}^{y}(e\Et_{y_0}e)=\Ef_y$, for regular $y$, so that the condition $\tilde f(y)\in \Phi_{y_0}^{y}(e\Et_{y_0}e)$ is trivially satisfied as is $\tilde f(y_0)\in \Et_{y_0}$. Hence $\Ct=\Cf$.  {Thus by Theorem \ref{thm-main} we have  $\Cf \subseteq \Ef$. On the other hand,
we saw that $\Ef$ is a subsemigroup of $\Cf$ and that the inclusion is continuous. 
Since $\Cf$ is compact this gives the result.}
\end{proof}

We end this section by establishing a criterion which implies the condition of the last corollary, namely that $\Ef_{y_0}=\Et_{y_0}$.
\begin{definition}\label{coincidence-definition}
 Let $\pi:(X,\sigma)\to (Y,\delta)$ be an equicontinuous factor. 
The {\em minimal rank} $r_\pi$ of the factor $\pi$ is the smallest possible cardinality $|\pi^{-1}(y)|$ of a fibre, $y\in Y$. 
The {\em coincidence rank} $cr_\pi(y)$ of the fibre $y\in Y$ is the largest possible cardinality a subset of $\pi^{-1}(y)$ can have, which contains only pairwise non-proximal elements. \end{definition}

If the system $(X,\sigma)$ is minimal, then the coincidence rank of
an equicontinuous factor can be shown to be independent of $y$ and so $cr_\pi=cr_\pi(y)$ is the coincidence rank of the factor $\pi:(X,\sigma)\to (Y,\delta)$. If the factor is not specified then the coincidence rank is meant to be the coincidence rank of the maximal equicontinuous factor. 
See \cite{Aujogue-Barge-Kellendonk-Lenz} for details and a context. 
Not every system contains regular fibres. It can be shown that for minimal systems with finite coincidence rank for the maximal equicontinuous factor, the maximal equicontinuous factor contains a regular fibre 
if and only if the system is {\em point distal } i.e.\ contains a point $x$ that is proximal only to itself \cite{Aujogue-Barge-Kellendonk-Lenz}, and if that is the case, any other equicontinuous factor must also contain regular fibres.
(Since this is  a side remark we don't include a proof.)
\begin{lem} If the minimal rank $r_\pi$ of
the equicontinuous factor $Y$ of a minimal system is finite and the factor contains some regular fibre then $y\in Y$ is regular if and only if $|\pi^{-1}(y)|=r_\pi$. 
\end{lem}
\begin{proof} 
Let $y_0$ be a regular point. Then $cr_\pi=|\pi^{-1}(y_0)|$. 
It follows that $cr_\pi \geq r_\pi$. On the other hand, since $r_\pi$ is finite there exists a point $y_1$ for which  $r_\pi=|\pi^{-1}(y_1)|$. Clearly $cr_\pi(y_1)\leq |\pi^{-1}(y_1)|$. Hence 
$cr_\pi=r_\pi$. Thus all points of a regular fibre must be pairwise non-proximal, and moreover, a fibre cannot contain more than $r_\pi$ pairwise non-proximal points.
\end{proof}

\begin{lem} \label{lem-idp} Let $\pi:(X,\sigma)\to (Y,\delta)$ be an equicontinuous factor with finite minimal rank.
Let $f\in \Ef$ be an element which acts on the singular fibres as an idempotent. 
Then for some $N$,  $f^N = f$  on the singular fibres and  $f^N=\Id$ on the regular fibres. 
\end{lem}
\begin{proof} Since regular fibres contain only distal points, and only finitely many, any element of $f\in \Ef$ must act on a regular fibre as a bijection. Since regular fibres have $r_\pi$ elements, then if  $N=r_\pi!\,$,
$f^N$ acts like the identity on a regular fibre. If $f$ acts like an idempotent on the singular fibre then $f^N$ acts like $f$ on the singular fibres. 
\end{proof}
We denote by $Y/T$ the space of $T$-orbits of $Y$ and its elements by $[y]$.
\begin{cor} \label{cor-main1}
Consider a unique singular orbit system with finite minimal rank. Let $y_0\in Y$ be singular. 
The restriction  $\Et_{y_0}$ of $\Et $ to $\pi^{-1}(y_0)$ contains all idempotents of $\Ef_{y_0}$. In particular, if  $\Ef_{y_0}$ is generated by its idempotents then $\Ef_{y_0}=\Et_{y_0}$ and consequently, 
$$\Ef = \Cf \cong \Ef_{y_0}\times \prod_{\stackrel{[y]\in Y/T}{y\neq y_0}} G_\pi.$$
This is a topological isomorphism if we equip the r.h.s.\ with the product topology.
\end{cor}
\begin{proof}
Any idempotent of $\Ef_{y_0}$ is the restriction of an element $f\in \Ef$ which, by Lemma~\ref{lem-idp}, may be assumed to act trivially on all regular fibres. Hence  any idempotent of $\Ef_{y_0}$ is the restriction of an element $f\in \Et$.
Under the assumption of finite minimal rank the structure group $G_\pi$ must be finite and thus topologically isomorphic to $\Ef_y$ for regular $y$. Covariance allows us to factor out the action of $T$ and thus describe $\Cf$ as a direct product over the space of orbits $Y/T$.
\end{proof}


\subsection{Recovering $E(X)$ from $E^{fib}(X)$ and $Y$}
Although we now have a pretty good description of $E^{fib}(X)$ and $Y$ for unique singular orbit systems with finite minimal rank, it is not obvious how this describes $E(X)$. As our interest lies in minimal systems which have a singular fibre, their Ellis semigroup must contain two non-commuting idempotents. This implies that $E(X)$ cannot be 
left-topological\footnote{Since $T$ is abelian, if left multiplication is continuous then $\lim_\nu \sigma^{t_\nu} \lim_\mu \sigma^{s_\mu} = \lim_\mu \lim_\nu \sigma^{t_\nu+s_\mu} = \lim_\mu \sigma^{s_\mu} \lim_\nu \sigma^{t_\nu}$, hence all elements of $E(X)$ commute.}, 
even when $E^{fib}(X)$ and $Y$ are topological. This is a sign that we cannot expect a semidirect product construction, paralleling that of groups, which describes $E(X)$ with its topology through $E^{fib}(X)$ and $Y$. However, on the purely algebraic side, we will see that Proposition \ref{extension-semigroup} turns out to be useful in this regard.

 \begin{notation}\label{not:structure groups}
 We let 
 $\M(X)$ denote the kernel of $E(X)$ and  $\M^{fib}(X)$ denote the kernel of $\Ef(X)$. Recall that these kernels are completely simple. Picking a minimal idempotent $e$, we let 
$\RSTp=eE(X)e$ and $\RSTfp=e\Ef(X) e$ denote the Rees structure group of $\M(X)$ and $\M^{fib}(X)$ respectively.
\end{notation}

As $Y=\tilde\pi(E(X)) = \tilde\pi(eE(X)) \subset \tilde\pi(\M(X))$, 
\eqref{eq-ext} gives rise to the exact sequence  
\begin{equation}\label{eq-kernel-ext}\nonumber
\Ef(X)\cap \M(X) \hookrightarrow \M(X) \stackrel{\tilde \pi} \twoheadrightarrow Y .
\end{equation}
$\Ef(X)\cap \M(X)$ contains all idempotents of $\M(X)$. Moreover, it is simple, as can be seen as follows: 
As $\M(X)$ is completely simple, given $x,y\in  \Ef(X)\cap \M(X)$ there is an idempotent $z \in \M(X)$ such that $x,z$ belong to the same minimal left, while $z,y$ belong to the same minimal right ideal of $\M(X)$. Since $z$ is an idempotent we have $z\in \Ef(X)$.
Since $x,z\in  \Ef(X)\cap \M(X)$ belong to the same minimal left ideal of $\M(X)$ then there is $a\in \M(X)$ such that $x=az$. It follows that $\tilde \pi(a) = 0$, thus $a\in  \Ef(X)\cap \M(X)$. Similarly, since $z,y\in  \Ef(X)\cap \M(X)$ belong to the same minimal right ideal of $\M(X)$ then there is $b\in \Ef(X)\cap \M(X)$ such that $z=yb$.  Hence $x=ayb$ for $a,b\in\Ef(X)\cap \M(X)$. This proves that $\Ef(X)\cap \M(X)$ is simple 
and therefore equal to the kernel $\M^{fib}(X)$ of $\Ef(X)$. 

A further restriction of \eqref{eq-ext} to $e \M(X) e$ leads to the exact sequence of groups
\begin{equation}\label{eq-kernel-ext1}\nonumber
\RSTfp \hookrightarrow \RSTp \stackrel{\tilde \pi} \twoheadrightarrow Y .
\end{equation}
We show now that for systems which satisfy the conclusion of  Corollary~\ref{cor-main} the above sequence has a split section, so that the structure group $\Gg$ is the semi-direct product of  $\RSTfp$ with the group $Y$.

\begin{prop} \label{prop-shomo}
Let $(X, \sigma)$ be a minimal 
system with an equicontinuous factor $\pi:(X,\sigma)\to (Y,\delta)$ such that $\Ef=\Cf$.  Let $e$ be a minimal idempotent of $E$ and  
$s:Y\to eEe$ be a lift of $\tilde \pi$ which satisfies $s(\delta^t(y)) = \sigma^t s(y)$ and 
$s(-y)=s(y)^{-1}$. Define, for $y\in Y$ the map $\hat s(y):eX\to eX$ by
$$\hat s(y)(x) = s(\pi(x)+y) s(\pi(x))^{-1} (x).$$ 
Then $\hat s:Y\to eEe$ is a right inverse to $\tilde\pi$ which 
satisfies $\hat s(\delta^t(y)) = \sigma^t \hat s(y)$ and is a group homomorphism.
\end{prop} 
\begin{proof} Let $y\in Y$.
By definition
$$\hat s(y)(x) =  s(y) g(\pi(x))  (x)$$ 
where $g(z)  = s(y)^{-1} s(z+y) s(z)^{-1}$.  
We see that $g(z)(x)\in e\pi^{-1}(z)$ for all $x\in e\pi^{-1}(z)$, 
hence $g(z):e\pi^{-1}(z)\to e\pi^{-1}(z)$ is an is element of $e\Ef_{z} e$.
Using $s(\delta^t(y)) = \sigma^t s(y)$ we obtain
$g(\delta^t(\pi(x)))(\sigma^t x) = \sigma^t g(\pi(x))(x)$. Thus
$Y\ni z\mapsto g(z)\in e\Ef_{z} e$ is covariant along the orbit of $y$ and hence an element of $e \Cf e$. By assumption  $e \Cf e=e \Ef e$.
Thus $\hat s(y) = \tilde s(y)g$ is an element of $eEe$. 

We show that $\hat s(y)$ is a right inverse to $\tilde\pi$. Let $x\in \pi^{-1}(0)$. We have
$$\tilde\pi \hat s(y) = \pi_*(\hat s(y))(0) = \pi(\hat s(y)(x)) = \pi(s(y)(x))=y+\pi(x) = y$$
where we have used $s(0)^{-1}=e$ in the third equality. 
The identity $\hat s(\delta^t(y)) = \sigma^t \hat s(y)$ follows readily.
It remains to show that $\hat s$ is multiplicative: 
\begin{eqnarray*}
\hat s(y_1+y_2)(x) &=&  s(\pi(x)+y_1+y_2) s(\pi(x))^{-1} (x)\\
&=&  s(\pi(x)+y_1+y_2)s(\pi(x)+y_2)^{-1} s(\pi(x)+y_2)s(\pi(x))^{-1} (x)\\
&=& \hat s(y_1)\hat s(y_2)(x).
\end{eqnarray*}
\end{proof}
\begin{cor}\label{cor:reproducing M}
Let $(X,T,\sigma)$ be a minimal 
system with an equicontinuous factor $\pi:(X,\sigma)\to (Y,\delta)$ such that $\Ef=\Cf$. 
Then the structure group  $\RSTfp $ of $\M^{fib}(X)$ is isomorphic to  $G_\pi^{Y/T}$. Moreover, the structure group of $\M(X)$ is $\RSTp=\RSTfp\rtimes Y$. Furthermore
if  $M[\RSTfp;I,\Lambda;A]$ is the normalised Rees matrix form for $\M^{fib}(X)$ w.r.t.\ $e$, then  $\M(X)$ is algebraically isomorphic to $M[\RSTfp\rtimes Y;I,\Lambda;A]$. 
\end{cor}
\begin{proof} Apply Proposition~\ref{extension-semigroup}, {Corollary \ref{cor-ff}} and Proposition~\ref{prop-shomo}. 
\end{proof}
For unique singular orbit systems with finite minimal rank and for which $\Ef_{y_0}(X)$ is generated by its idempotents, we have now reduced the calculation of the kernel of their Ellis semigroup to the calculation of the kernel of $\Ef_{y_0}(X)$ which we denote $\M^{fib}_{y_0}(X)$. Indeed, if $M[G;I,\Lambda;A]$ is the normalised matrix form of $\M^{fib}_{y_0}(X)$ w.r.t.\ $e$, then $G=G_\pi$ which we may identify with the subgroup $G_\pi\times \prod_{[y_0]\neq [y]\in Y/ T}\{\one\}$ of $G_\pi^{Y/T}$. It follows that
\begin{equation}\label{eq-alter}
\M^{fib}(X) \cong   M[G_\pi;I,\Lambda;A]\times \prod_{[y_0]\neq [y]\in Y/ T}G_\pi \cong M[G_\pi^{Y/T};I,\Lambda;A]
\end{equation}
where the second topological isomorphism is given by the map $((i,g,\lambda),f)\mapsto 
(i,(g,f),\lambda)$.
We will see in the next section how to compute $G_\pi$, $I$, $\Lambda$, and $A$
for systems arising from bijective substitutions.

\section{Bijective substitutions and their Ellis semigroup}\label{bijective}

In this section we discuss the Ellis semigroup of a family of minimal $\Z$-actions which are both unique singular orbit systems, and also systems  for which forward/backward proximality is non-trivial and agrees with forward/backward asymptoticity. This is the family of bijective constant length substitution shifts. 
For these systems, Corollary \ref{cor:prox=asym}  tells us that  $E(X)$ is the disjoint union of its kernel $\M(X)$ with the acting group $\Z$, so that a description of 
$\M(X)$ suffices to completely describe the Ellis semigroup. Next,
for most of these systems, Corollary \ref{cor:reproducing M} will apply, so that we are on the way to describing $\M(X)$ once we know its restriction to a singular fibre which we call below the structural semigroup. This is the content of Theorem \ref{thm-RMG}. We consolidate to get a global statement  in Theorem \ref{thm-main2}.
Finally, we identify the substitution shifts to which we cannot apply Corollary \ref{cor:reproducing M}, and we replace it with Theorem \ref{thm-main4}.

\subsection{Generalities}\label{generalities} We briefly summarise the notation and results concerning substitutions that we will need; for an extensive background see \cite{Baake-Grimm} or \cite{Pytheas-Fogg}.

A  {\em substitution} is a map from a finite set $\mathcal A$, the alphabet, to the set of 
nonempty finite words (finite sequences) on $\mathcal A$. We extend $\theta$ to a map on finite words by concatenation: 
\begin{equation}\label{eq-concat}\theta(a_1\cdots a_k) = \theta(a_1)\cdots \theta(a_k),
\end{equation}
and to bi-infinite sequences $\cdots u_{-2} u_{-1} u_0 u_1 \cdots$  as  
\[\theta   (\cdots u_{-2} u_{-1} u_0 u_1 \cdots ) :=   \cdots \theta (u_{-2}) \theta (  u_{-1}    ) \theta  (u_{-1})\cdot   \theta (     u_{0})  \theta (    u_{1}) \cdots \, .\]
Here the $\cdot$ indicates the position between the negative indices and the nonnegative indices. 

We say that $\theta$ is {\em primitive} if there is some 
$k\in \N$ such that for any $a,a'\in \mathcal A$,
the word $\theta^k(a)$ contains at least one  occurrence of $a'$. 
We say that a finite word is {\em allowed} for $\theta$ if it appears somewhere in $\theta^k(a)$ for some $a\in \Aa$ and some  $k\in\N$.

The {\em substitution shift} $( X_\theta,  \sigma)$ is the dynamical system where the space $X_\theta$  consists of all bi-infinite sequences 
all of whose subwords are allowed for $\theta$. If $\theta$ is primitive, $X_\theta=X_{\theta^n}$ for each $n\in \N$.
We equip $X_\theta$ with the subspace topology of the product topology on $\Aa^\Z$, making the left shift map $\sigma$  a continuous  $\Z$-action. Primitivity  of $\theta$ implies that $(X_\theta,\sigma)$ is minimal.

We say that a primitive substitution is {\em aperiodic} if $X_\theta$ does not contain any $\sigma$-periodic sequences. This is the case if and only if $X_\theta$ is an infinite space. 
The substitution $\theta$ has
{\em (constant) length~$\ell$}  if for each $a\in \mathcal A$, 
$\theta (a)$ is a word of length $\ell$. In this case one can describe the substitution with  $\ell$ maps $\theta_i:\mathcal A \rightarrow \mathcal A$, $0\leq i \leq \ell-1$, such that
\begin{equation}\label{eq-as-perm}\nonumber
\theta(a) = \theta_0(a)\cdots \theta_{\ell-1}(a)
\end{equation}
for all $a\in\Aa$.
A substitution $\theta$ is {\em bijective} if it has constant length and each of the maps $\theta_i$ is a  bijection. 
If $\theta$ is bijective, then $X_\theta$ is the disjoint union of finitely many primitive bijective substitution shifts, and consequently its Ellis semigroup is also the disjoint union of finitely many Ellis semigroups of primitive substitution shifts.  Henceforth we  assume that $\theta$  is primitive  but this comment means that  all our results have analogous statements for non-primitive bijective substitutions.
 
 We say that {the bijective} $\theta$ is {\em simplified} if
 \begin{enumerate}
 \item every $\theta$-periodic point is a fixed point of $\theta$, so that in particular $\theta_0=\theta_{\ell-1} = \one$, and
 \item each word $\theta(a)$ contains all letters from $\mathcal A$.
  \end{enumerate}
Given any bijective substitution $\theta$, both properties will be satisfied by a large enough power $\theta^n$ of $\theta$. Indeed,  if $M$ is the lowest common multiple of the least periods of the periodic points, then each periodic point is a fixed point under $\theta^M$. Since for any $n\in \N$, $X_\theta= X_{\theta^n}$, there will be no loss in generality in assuming that $\theta$ is simplified and this is henceforth a standing assumption.

\subsection{An equicontinuous factor with a unique orbit of singular fibres}
Let $\theta$ be an aperiodic primitive substitution of length $\ell$. Define
$B^{(n)}:=\theta^n(X_\theta),$
which is a clopen subset of $X_\theta$. Then
$\sigma^i(B^{(n)})= \sigma^j(B^{(n)})$ if $i-j=0\: \mbox{\rm mod }\ell^n$ whereas
otherwise $\sigma^i(B^{(n)})\cap \sigma^j(B^{(n)})=\emptyset$ \cite[Lemma II.7]{dekking}. In other words 
 \[ \mathcal P_n = \{ \sigma^k( B^{(n)}) : 0\leq k\leq \ell^n-1   \} \]
is a $\sigma^{\ell^n}$-cyclic partition  of $X_\theta$ of size $\ell^n$ 
 For $n\geq1$, define
$ \pi_n : X_\theta \to \Z/\ell^n\Z$ by
$$\pi_n(x) = i \quad\mbox{\rm if }x\in \sigma^{i} (B^{(n)}).$$ 
The map $\pi_n$ can be described as follows. Using the partition $\mathcal P_1$,  any bi-infinite sequence $x=(x_i)_{i\in\Z}\in X_\theta$ can be uniquely decomposed into blocks of length $\ell$ such that
\begin{itemize}
\item[{(i)}] The $i$-th block is a substitution word $\theta(a_i)$, for some $a_i\in \Aa$. Here we say that the $0$-th block is the one which contains $x_0$, and
\item[{(ii)}] The sequence $(a_i)_{i\in\Z}$ is an element of  $X_\theta$.
\end{itemize}
Now set $\pi_1(x):=i$ if the $0$-th block starts at index $-i$ (if we shift that block $i$ units to the right then its first letter has index $0$). 
This procedure can be performed with $\mathcal P_n$ and $\theta^n$ yielding an analogous definition for 
$\pi_n(x)$. In particular, the $\pi_n$ are pattern equivariant (or local) and hence continuous. 
Note that if  $\pi_n(x) = i$, then $\pi_{n+1}(x) \equiv i \mod \ell^{n}$. Therefore, the collection of these maps $\pi_n$ defines a continuous map 
\begin{equation}\label{ef-map} \pi : X_\theta \to Y:=\lim_{\leftarrow} \Z/\ell^n\Z \end{equation}
onto the inverse limit $\lim_{\leftarrow} \Z/\ell^n\Z$ defined by the canonical projections $\Z/\ell^{n+1}\Z \twoheadrightarrow \Z/\ell^n\Z$. The inverse limit space can be identified with 
the space of left-sided sequences $(y_i)_{i<0} = \cdots y_{-2}\, y_{-1}$, $0\leq y_i<\ell$, and then $\pi(x)= (y_i)_{i<0}$ is such that for each  positive integer
$n$, $\pi_{n}(x)=\sum_{i=-n}^{-1} \ell^{-i-1}y_{i}$. 
It then follows that $\pi\circ\sigma = \add\circ \pi$ where $\add$ is addition of $1 =  \cdots 00 1$ (only the last digit is not $0$) with carry over. 
Its additive inverse is addition of  $-1= \cdots \ell\!-\!1\, \ell\!-\!1\, \ell\!-\!1\,$. 
In other words $(X_\theta,\sigma)$ factors onto the odometer with $\ell$ digits (adding machine).
This is the equicontinuous factor map with which we work.
As the space is the space of $\ell$-adic integers, we will denote it using the notation $\Z_\ell$. 

\begin{prop}\label{prop-sf}
Let $\theta$ be a primitive aperiodic bijective (and simplified) 
substitution of length $\ell$ and $\pi:X_\theta\to\Zl$ be defined by \eqref{ef-map}.
The fibre $\pi^{-1}(0)$ contains exactly the $\theta$-fixed 
points. These are in one-to-one correspondence with the allowed two letter words. 
\end{prop}
\begin{proof}
It is quickly seen that $\pi\circ \theta = (\times \ell)\circ \pi$ where $(\times \ell)$ is multiplication by $\ell$ in $\Zl$ and corresponds to left shift with adjoining a $0$:
$(\times\ell)(\cdots y_{-2}\, y_{-1}) = \cdots y_{-2}\, y_{-1} 0$. Hence any $\theta$-fixed point is mapped by $\pi$ to a $(\times \ell)$-fixed point in $\Z_\ell$, and the only such one is $0$.
Thus all $\theta$-fixed points belong to $\pi^{-1}(0)$.
It also follows that $\theta$ must preserve $\pi^{-1}(0)$. 

Since the maps $\theta_i$ are bijections of $\Aa$, $\theta$ is injective on $X_\theta$. Hence it is injective on $\pi^{-1}(0)$. 
We claim that $\pi^{-1}(0)$ must be finite. To prove the claim 
let $x,x'\in X_\theta$. If $x\neq x'$ there exists $n\in\N$ such that $x_{[-\ell^n,\ell^n-1]} \neq x'_{[-\ell^n,\ell^n-1]}$
(here $x_{[n,m]}$ is the word $x_n x_{n+1}\cdots x_m$). It follows that $\theta^{-n}(x)_{[-1,0]} \neq \theta^{-n}(x')_{[-1,0]}$. 
Since there are only finitely many words of length two  
$\pi^{-1}(0)$ must be finite. It follows that the restriction of $\theta$ to $\pi^{-1}(0)$ is bijective and thus $\pi^{-1}(0)$ must be a union of periodic orbits under $\theta$.

As $\theta$ is bijective and simplified we have $\theta(x)_{[-1,0]} = x_{[-1,0]}$ for any $x\in X_\theta$. It follows that $\theta$-periodic points are $\theta$-fixed points and that they are in  
one-to-one correspondence with  the allowed two letter words.
\end{proof}

\begin{prop}\label{one-singular-orbit}
Let $\theta$ be a primitive aperiodic bijective substitution of length $\ell$ and let 
$\pi:X_\theta\to\Zl$ be defined by \eqref{ef-map}.
Then the orbit of $\pi^{-1}(0)$ is the only singular fibre orbit. The minimal rank is 
$r_\pi = s$ where $s$ is the size of the alphabet.  
\end{prop}
\begin{proof}
Suppose that $y=\ldots y_{-2} y_{-1}$ does not belong to the $\Z$-orbit of $0$. 
This is the case precisely if for infinitely many $n$,  $y_{-n}\neq 0$ and, for infinitely many $n$,  $y_{-n}\neq \ell-1$. Now if we take $x\in \pi^{-1}(y)$ and decompose it into  substitution words $\theta^n(a)$ of level $n$ (as described above), 
then the substitution word $\theta^n(a_0)$ which covers index $0$ must be 
$\theta^n(a_0) = x_{[k_n,k_n+\ell^n-1]}$ where $k_n = -\sum_{i=-n}^{-1}\ell^{-i-1}y_{i}$. 
Since $y_{-n}\neq 0$ for infinitely many $n$ we have $k_n\stackrel{n\to \infty} \longrightarrow -\infty$, and 
since $y_{-n}\neq \ell-1$ for infinitely many $n$ we have $k_n + \ell^n-1\stackrel{n\to \infty} \longrightarrow=+\infty$.
Furthermore, 
by bijectivity of $\theta$, $a_0$ is uniquely determined by $x_0$. 
It follows that $x$ is uniquely determined by $y$ and $x_0$. Since there are exactly $s$ choices for $x_0$ we see that $\pi^{-1}(y)$ contains $s$ elements.

We now show that $\pi^{-1}(y)$ is a regular fibre if $y$ does not belong to the orbit of $\Z$. 
Suppose that $x,x'$ were proximal. Then there exists $n\in\Z$ such that $x_n=x'_n$. In other words $\sigma^n(x)_0 = \sigma^n(x')_0$. Also $y+n$ does not belong to the $\Z$-orbit of $0$ and since $\sigma^n(x),\sigma^n(x')\in \pi^{-1}(y+n)$ we conclude from the above that $x=x'$. Hence all points of $\pi^{-1}(y)$ are pairwise non-proximal. 

We have seen above that $\pi^{-1}(0)$ has $s^{(2)}$ elements where $s^{(2)}$ is the number of allowed two letter words. Given that $\theta$ is aperiodic we must have $s^{(2)}>s$. Thus 
$\pi^{-1}(0)$ cannot be a regular fibre.
\end{proof}
Since $\theta$ is simplified its fixed points are precisely those of the form $\theta^{\infty}(a)\cdot\theta^{\infty}(b)$, where $ab$ is an allowed word for $\theta$. Such a fixed point is uniquely determined by $ab$.   
We will use the notation $a\cdot b$ to denote it. 
\begin{cor}
Let $\theta$ be a primitive bijective aperiodic substitution of constant length. If two points $x,x'\in X_\theta$ are forward (or backward) proximal then they are forward (or backward) asymptotic. Furthermore, forward and backward proximality are non-trivial.
In particular the Ellis semigroup $E(X_\theta)$ is completely regular and the disjoint union of its kernel $\M(X_\theta)$ with $\Z$.
\end{cor}
\begin{proof}
Suppose that $x,x'\in X_\theta$ are forward proximal. Then they are proximal and so by Propositions ~\ref{one-singular-orbit} and \ref{prop-sf} there is $n\in\Z$ such that $\sigma^n(x)$ and $\sigma^n(x')$ are fixed points of the (simplified) substitution. Since they are forward proximal and $\sigma$ is left shift this means that there are allowed two-letter words $ba,b'a$ such that  $\sigma^n(x)=b\cdot a$ and $\sigma^n(x')=b'\cdot a$. In particular the two sequences agree to the right and hence are forward asymptotic. If forward proximality were trivial then every two letter word would be determined by its right letter. This cannot be the case as the substitution is aperiodic. Hence forward proximality is non-trivial. For the backward motion we argue in a similar way. The result now follows from Corollary \ref{cor:prox=asym}.
\end{proof}

\subsection{The kernel $\M^{fib}_0(X_\theta)$}\label{structural-section}
The restriction $E^{fib}_0(X_\theta)$ of $\Ef(X_\theta)$ to the singular fibre $\pi^{-1}(0)$ of the factor map $\pi:X_\theta\rightarrow \Z_\ell$ contains besides the identity map only its kernel $\M^{fib}_0(X_\theta)$. This kernel is a finite semigroup which we now compute. We also call it the {\em structural semigroup} of $\theta$. It completely determines $\Ef(X_\theta)$. 

\bigskip

Recall that since $\theta$ has length $\ell$ there are maps  $\theta_i:\mathcal A \rightarrow \mathcal A$ such that
$
\theta(a) = \theta_0(a)\cdots \theta_{\ell-1}(a)
$
for all $a\in\Aa$.
$\theta$ is thus uniquely determined by what we call its {\em expansion}, namely its representation as a concatenation of $\ell$ maps, which we write as
$$\theta = \theta_0|\theta_1|\cdots |\theta_{\ell-1}.$$
It follows from (\ref{eq-concat}) that 
the composition of two substitutions $\theta$, $\theta'$ of length $\ell$ and $\ell'$ over the same alphabet (which we simply denote by $\theta \theta'$) has then an expansion into
$\ell \ell'$ maps
\begin{equation}\nonumber
\theta \theta' = \theta_0\theta'_0 |\cdots |\theta_{\ell-1}\theta'_0|\theta_0\theta'_1| \cdots |
\theta_{\ell-1}\theta'_{\ell'-1} 
\end{equation}
where the product $\theta_i\theta'_j$ is that of permutations, that is, composition of bijections. 
In particular, the expansion of $\theta^2$ is given by 
\begin{equation}\label{expansion}
(\theta^2)_0|\cdots |(\theta^2)_{\ell^2-1} = \theta_0\theta_0 | \cdots |\theta_{\ell - 1}\theta_{0}|  \theta_0\theta_1 |  \cdots |\theta_{\ell-1}\theta_{\ell-1} 
\end{equation}
and iteratively we find, for any given $n$ the $\ell^n$ bijections $(\theta^n)_i$ corresponding to the expansion of $\theta^n$. 

\begin{definition}\label{def:structure group substitution} Given a bijective substitution $\theta$, we define the
{\em structure group} $\Gstr$ of $\theta$ to be the group generated by all the bijections 
$(\theta^n)_i$, $n\in\N$, $i=0,\cdots,\ell^n-1$, and its {\em R-set} by 
$$\rset:= \{(\theta^n)_{i}(\theta^n)_{i-1}^{\,-1}\in  G_\theta : n\in\N,i=1,\cdots,\ell^{n}-1\}.$$
\end{definition}
Note that $\rset$ is the collection of bijections we need to apply (from the left) to go from some element $(\theta^n)_{i-1}$ in the expansion of some power of the substitution to its successor $(\theta^n)_{i}$. The name $R$-set is motivated by the fact that  $\rset$ will label the right ideals of $\Sfib$. 

\begin{lem}\label{I-is-everything} 
If $\theta$ is simplified, then $\Gstr$ is generated by $\rset$ and 
$$\rset = \{\theta_{i}\theta_{i-1}^{\,-1}\in G_\theta : i=1,\cdots,\ell-1\}.$$
\end{lem}
\begin{proof}
The first statement follows recursively as $\theta_i = \theta_i \theta_{i-1}^{-1} \theta_{i-1}$ and $\theta_0=\one$.
We prove the  second statement for $n=2$ 
as the general statement then follows by induction. Let  $(\theta^2)_{i-1}(\theta^2)_{i}$ be two consecutive bijections in the expansion of $\theta^2$. We consider  two cases, the first if 
$(\theta^2)_{i-1}(\theta^2)_{i}$ appears as two consecutive columns in a single substitution word, the second if it lies on the boundary, across two substitution words.

In the first case, $(\theta^2)_{i-1} | (\theta^2)_{i} = \theta_{j-1}  \theta_k |\theta_{j}  \theta_k  $ for some $j\leq \ell - 1$ and some $0\leq k \leq \ell -1$, as in Equation \eqref{expansion}. But then 
\[  (\theta^2)_{i}  (\theta^2)_{i-1}^{\,-1} =   \theta_{j}  \theta_k     (\theta_{j-1}  \theta_k)^{-1} =   \theta_{j}  (\theta_{j-1})^{-1},   \]
and we are done as this last expression belongs to $I$.

In the second case, $(\theta^2)_{i-1} | (\theta^2)_{i} = \theta_{\ell-1}  \theta_k |\theta_{0}  \theta_{k+1}$ for some $k<\ell -1$. Since $\theta$ is simplified, $\theta_0=\theta_{\ell -1}= \one$, and here also we are done.
\end{proof}
In this section we will prove Theorem \ref{thm-RMG}; in its statement the semigroup has sandwich matrix as defined in \eqref{eq:matrix}.

\begin{thm}\label{thm-RMG}
Let $\theta$ be a primitive aperiodic bijective substitution.
The \Sstr\ $\Sfib$ of $\theta$ is isomorphic to the matrix semigroup $M[ \Gstr, \rset, \{ \pm\}; A]$ where 
$\Gstr$ is the structure group and 
$\rset$ is the R-set of $\theta$. 
\end{thm}

We will first give a description of $\Sfib$ as a subsemigroup of $F(\pi^{-1}(0))$, the set of functions from $\pi^{-1}(0)$ to itself,  and then compute its Rees matrix form. 

We
recall that $\pi^{-1}(0)$ is the set of fixed points of $\theta$ which we denote $a\cdot b$ where $ab$ is an allowed two-letter word of $\Aa$. 
To describe the action of $E^{fib}(X_\theta)$ on such a fixed point $a\cdot b$ we consider the set $\tc$ of all possible pairs of consecutive permutations $(\theta^n)_{i-1},(\theta^n)_{i}$, occurring in $\theta^n$, $n\in\N$, $i=1,\cdots \ell^n-1$. We write them with a dot 
$(\theta^n)_{i-1}\cdot (\theta^n)_{i}$, or abstractly $L \cdot R$. We note that the R-set is related to $\tc$, namely
$$ \rset = \{RL^{-1}: L \cdot R\in \tc\}.$$
Notice also that $\tc$ is the same for any power of the substitution. \begin{lem}\label{lem-inv}
$\tc$ is invariant under the right 
diagonal $G_\theta$-action $$(L \cdot R)g = (L g \cdot R g).$$ \end{lem}
\begin{proof} 
Suppose that $L \cdot R\in G^{(2)}$, then for some $k\in \N$ and $0<i< \ell^k$ it appears as  $L \cdot R= (\theta^k)_{i-1}\cdot(\theta^k)_{i}$. Let $g\in G_\theta$,  so that it appears as    $g=(\theta^n)_j$ for some $n\in \N $ and some $0\leq j\leq \ell^n -1$. Then, using the expansion of $\theta^n$ obtained in 
Equation \eqref{expansion},
we find that $L g\cdot  R  g$ 
appears as two consecutive columns in the expansion of $ \theta^{n+k}$ and hence   belongs to $\tc$
 for each $g\in G_\theta$.
 \end{proof}
As we assume that  $\theta$ is simplified,  we have 
$(\theta^k)_0 = (\theta^k)_{\ell^k -1} = \one$ for each $k\geq 1$.  Then
$$\theta^k(L|R) = (\theta^k)_0L |\cdots |(\theta^k)_{\ell^k-2} L | L|R | (\theta^k)_1 R|\cdots| 
(\theta^k)_{\ell^k-1} R$$
where we have used that $(\theta^k)_{\ell^k -1}L=L$ and $(\theta^k)_{0}R = R$.
Hence if $L \cdot R = (\theta^n)_{i-1}\cdot(\theta^n)_{i}$ then the expansion of 
$\theta^k\theta^n$ contains $\theta^k(\theta^n)_{i-1}| \theta^k(\theta^n)_{i}$ at positions 
$[\ell^{k}(i-2),\ell^{k}i-1]$.
\begin{prop} \label{prop-hin}
Let $L \cdot R\in\tc$. Then 
$E^{fib}$ contains an element   
$f_{[L \cdot R;+]}$
 which acts on the singular fibre $\pi^{-1}(0)$ as 
$$f_{[L \cdot R;+]}(a\cdot b) = L(b)\cdot R(b),$$
and it contains an element    $ f_{[L \cdot R;-]}$ which acts on this fibre as
$$\quad f_{[L \cdot R;-]}(a\cdot b) = L(a)\cdot R(a).$$
\end{prop}
\begin{proof}
Recall we assume that $\theta$ is simplified, so 
 that $\theta_0=\theta_{\ell-1}=\one$.   

Let $n$ be such that $L \cdot R = (\theta^n)_{\nu-1}\cdot(\theta^n)_\nu$, for some $1\leq \nu\leq \ell^n-1$. 
Let 
$a\cdot b$ be a fixed point. 
Then the two-letter word $\sigma^{\nu}(a\cdot b)_{[-1,0]}$ is $(\theta^n)_{\nu-1}(b)(\theta^n)_\nu(b)$. Furthermore 
the expansion of 
$\theta^k\theta^n$ contains $\theta^k(\theta^n)_{i-1}| \theta^k(\theta^n)_{i}$ at positions 
$[\ell^{k}(i-2),\ell^{k}i-1]$.
Hence 
$$\sigma^{\nu\ell^k}(a\cdot b)_{[-\ell^k,\ell^k-1]} = \theta^k L(b)\,\theta^kR(b).$$ 
It follows that    
$$\sigma^{\nu \ell^k}(a\cdot b)   \stackrel{k\to +\infty}\longrightarrow L(b)\cdot R(b)$$
in the topology of $X_\theta$.
By compactness there exists $f_{[L \cdot R;+]}\in E(X_\theta)$ which agrees with the map 
$a\cdot b\mapsto L(b)\cdot R(b)$ on the singular fibre. It follows from the exact sequence (\ref{eq-ext}) that an element of $E(X_\theta)$ either preserves all $\pi$-fibres or none. Hence 
$f_{[L \cdot R;+]}\in E(X_\theta)^{fib}$.

To construct elements in $E(X_\theta)$ which acts like 
$a\cdot b\mapsto L(a)\cdot R(a)$ on the singular fibre we take 
$\nu' = \nu-\ell^{n}$ with $n$ and $\nu$ as above. Then the two-letter word $\sigma^{\nu'}(a\cdot b)_{[-1,0]}$ is $L(a)\,R(a)$ and, similarly we find 
$$\sigma^{\nu' \ell^k}(a\cdot b)   \stackrel{k\to +\infty}\longrightarrow L(a)\cdot R(a).$$
By compactness we find the required map $ f_{[L \cdot R;-]}$.
\end{proof}
Let us denote the restriction of $f_{[L \cdot R;\epsilon]}$ to $\pi^{-1}(0)$ by $[L \cdot R;\epsilon]$ and set
$$\tc_\pm := \{[L\cdot R;\epsilon]: L\cdot R\in \tc, \epsilon \in\{\pm\}\}$$
It is easily checked that different elements of  $\tc_\pm$ define different maps on $\pi^{-1}(0)$.

\begin{prop}\label{prop-rueck}
For any $\varphi\in \Sfib$ there exists $L \cdot R\in\tc$ and $\epsilon=\pm$ such that $\varphi = [L \cdot R;\epsilon]$. 
\end{prop}
\begin{proof} Any $\varphi\in \Sfib\subset \Ef_0(X_\theta)$ is the restriction of a function $f\in \Ef(X_\theta)$ to $\pi^{-1}(0)$. We consider first the case that this function belongs to $E(X,\Z^+)$,
that is, $f$ is a pointwise limit of a generalised sequence $(\sigma^{\nu_i})_i$ with $\nu_i> 0$ (the inequality is strict as $f\neq \Id$). Note that $\varphi(a\cdot b)_{[-1,0]}$ is an open neighbourhood of $\varphi(a\cdot b)$. Thus given $a\cdot b\in\pi^{-1}(0)$ there exists a 
$i_0$ such that $\varphi(a\cdot b)_{[-1,0]} = \sigma^{\nu_i}(a\cdot b)_{[-1,0]}$ for $i\geq i_0$. As $\pi^{-1}(0)$ is finite there exists a 
$\nu>0$ such that $\varphi(a\cdot b)_{[-1,0]} = \sigma^{\nu}(a\cdot b)_{[-1,0]}$  for all $a\cdot b\in\pi^{-1}(0)$. Let 
$L=(\theta^n)_{\nu-1}$ and $R=(\theta^n)_{\nu}$ where $\ell^n\geq \nu$. Then, for all  $a\cdot b\in\pi^{-1}(0)$ we have $\varphi(a\cdot b)_{[-1,0]} = L(b)\cdot R(b)_{[-1,0]}$. Since the fixed points are uniquely determined by their two-letter word on $[-1,0]$ the map $\varphi$ is given by $a\cdot b\mapsto L(b)\cdot R(b)$. It is thus the restriction of $f_{[L\cdot R;+]}$ to $\pi^{-1}(0)$.

If $f \in E^-(X_\theta)$ we argue similarly:
there exists a $\nu<0$ such that, for all $a\cdot b\in\pi^{-1}(0)$ we have $\varphi(a\cdot b)_{[-1,0]} = \sigma^{\nu}(a\cdot b)_{[-1,0]}$. Then we take $L=(\theta^n)_{\ell^n-\nu-1}$ and $R=(\theta^n)_{\ell^n-\nu}$ where $\ell^n\geq \nu$. This leads to  $\varphi(a\cdot b)_{[-1,0]} = L(a)\cdot R(a)_{[-1,0]}$, for all  $a\cdot b\in\pi^{-1}(0)$ and we conclude that $\varphi$ is the restriction of $f_{[L\cdot R;-]}$ to $\pi^{-1}(0)$.
\end{proof}
We can compute the compositions of elements of $\tc_\pm$, for example 
\begin{eqnarray*}
 {[L \cdot R;+]}   {{[L'\cdot R';+]}}(a\cdot b) &= & {[L \cdot R;+]}(L'(b)\cdot  R'(b))\\
& =& LR'(b)\cdot  RR'(b) \\
& = &   {[L{R'},R{R'},+]}(a\cdot b)
\end{eqnarray*}
 and likewise 
\begin{eqnarray*}
 {[L \cdot R;+]}  {{[L'\cdot R';-]}}(a\cdot b) 
&= & {[L \cdot R;+]}(L'(a)\cdot  R'(a))\\
& =& LR'(a)\cdot  RR'(a)\\
& = &  {[LR' \cdot R{R'};-]}(a\cdot b).
\end{eqnarray*}
In this way we get
\begin{cor}\label{cor-G2}
$\Sfib = \tc_\pm$ 
with product given by
\begin{eqnarray*}
 {[L \cdot R;+]} {[L'\cdot R';+]} &=& {[LR' \cdot RR';+]} \\
 {[L \cdot R;-]}  {[L'\cdot R';-]} &=& {[LL'\cdot RL';-]} \\
 {[L \cdot R;+]} {[L'\cdot R';-]} &=& {[LR' \cdot RR';-]} \\
 {[L \cdot R;-]} {[L'\cdot R';+]} &=& {[LL'\cdot R{L'};+]} .
\end{eqnarray*}
\end{cor}
\begin{proof}
 {Combine Propositions~\ref{prop-hin} and ~\ref{prop-rueck} 
together with the fact that all $[L \cdot R;\epsilon]$ act differently 
to see that there is a one-to-one correspondence between the elements of $\Sfib$ and $\tc_\pm$. The form of the product is a direct calculation along the lines above.}
\end{proof}

\begin{proof}[Proof of Theorem~\ref{thm-RMG}]
Given the result of Corollary~\ref{cor-G2} it remains to show that $\tc_\pm$ is isomorphic to $M[ \Gstr, \rset, \{ \pm\}; A]$ for a (any) fixed  choice of $g_0=R_0L_0^{-1}\in I_\theta$. Recall that, as a set, $M[ \Gstr, \rset, \{ \pm\}; A]= \rset\times\Gstr\times \{\pm\}$. 
Consider the map 
\begin{equation}\label{eq-bij}
	\tc_\pm\ni  [L \cdot R;\epsilon] \mapsto 
	\begin{cases}
		(RL^{-1} , RL_0R_0^{-1},\epsilon) \	& \text{if $\epsilon = -$ } \\
		(RL^{-1} , R,\epsilon)	& \text{if $\epsilon=+$}.
	\end{cases}
\end{equation}
Its injectivity is clear and its surjectivity is equivalent to Lemma~\ref{lem-inv}. 
A direct calculation shows that it preserves the product structures. 
\end{proof}
The Rees structure group of $M[ \Gstr, \rset, \{ \pm\}; A]$ is  $G_\theta$, the structure group of $\theta$. It is related to the structure group $G_\pi$ of the factor 
$\pi:(X_\theta,\sigma)\to (\Z_\ell,(+1))$ as follows: The choice of $g_0\in \rset$ to define the normalised matrix semigroup corresponds to a choice of
a minimal idempotent $e$ of $\Ef_0(X_\theta,\Z^+)$. Then we may view $G_\pi=e \Ef_0 e$. $G_\pi$ is thus a permutation of $e\pi^{-1}(0)$ and if we identify $e\pi^{-1}(0)$ via 
$a\cdot b\mapsto b$ with $\Aa$ then we obtain $G_\theta$. 

We remark that an idempotent $e\in \Ef_0(X_\theta,\Z^+)$ has the form $e=(i,\one,+)$, $i\in\rset$, 
so that $e \Sfib e = \{i\}\times G_\theta \times \{+\}$ which is isomorphic to $G_\theta$ via the projection onto the middle component.

\begin{example} \label{Martin}
Consider the substitution $\theta$ given by
\[
\begin{array}{c} a\\ b\\ c \end{array} 
\mapsto
\begin{array}{c} a\\ b\\ c \end{array}
\!\!\!\!\!\!{\begin{array}{c} b\\ a\\ c \end{array}}
\!\!\!\!\!\!{\begin{array}{c} a\\ c\\ b \end{array}}
\!\!\!\!\!\!{\begin{array}{c} a\\ b\\ c \end{array}}
\]
We use the notation $\begin{pmatrix}
           \alpha \\
           \beta \\
           \gamma
          \end{pmatrix} $
to denote the bijection that sends $a$ to $\alpha$, $b$ to $\beta$ and $c$ to $\gamma$.
The expansion of $\theta$ is $\theta_0|\theta_1|\theta_2|\theta_3$ with $\theta_0=\theta_3 = \one = \begin{pmatrix} a\\b\\c \end{pmatrix}$, while $\theta_1 =  \begin{pmatrix} b\\a\\c \end{pmatrix}$ and $\theta_2 =  \begin{pmatrix} a\\c\\b \end{pmatrix}$ are two transpositions.
We quickly find that 
\[ \rset = \left\{ \theta_1\theta^{-1}_0 = \begin{pmatrix}
           b \\
           a \\
           c
          \end{pmatrix},
   \theta_2\theta^{-1}_1   =   \begin{pmatrix} c\\a\\b \end{pmatrix}    ,
      \theta_3\theta^{-1}_2   =
    \begin{pmatrix} a\\c\\b \end{pmatrix}     \right\}.
   \] 
Clearly $\rset$ generates $\Gstr=S_3$. The normalised sandwich matrix is
$$ A \begin{pmatrix} \one & \one & \one \\ \one & \tau_1 & \tau_2 \end{pmatrix} $$
where $\tau_1 = \theta_1 \theta^{-1}_0\theta_1 \theta_2^{-1}=\theta_2$ and $\tau_2= \theta_2 \theta^{-1}_1\theta_2 \theta_3^{-1}=\theta_2 \theta_1\theta_2$ are transpositions. 
$\Sfib=M[S_3, \rset, \{ \pm\}; A]$ has 2 minimal left ideals each of which contains 18 elements, and 3 minimal right ideals each of which contains $12$ elements.

We note that $(X_\theta,\sigma)$ is not an {\em AI-extension} of $(\Z_4,+1)$; this can be 
seen by applying Martin's criterion, which we do not describe here (see \cite{martin} or \cite{Staynova} for definitions and details); it suffices to note that $\theta$ admits seven two-letter words,  has height one, and  the  set of two-letter words cannot be partitioned into sets of size three, creating an obstruction. Thus the techniques of \cite{Staynova} do not apply.

\end{example}

\subsection{The full semigroup $E(X_\theta)$} 
We now combine Theorem~\ref{thm-main} with
our results from Section~\ref{fibre-preserving} to determine $\Ef(X_\theta)$ and, as far as possible, $E(X_\theta)$. 
 We consider first the simpler case in which $\Sfib$ is generated by its idempotents.

\begin{thm}\label{thm-main2}
Let $\theta$ be a primitive aperiodic bijective substitution and suppose that its structural semigroup $\Sfib$ is generated by its idempotents.  
Then $\Ef(X_\theta)$ is topologically isomorphic to 
$$\Ef(X_\theta) \cong (\Sfib \cup \{\Id\}) \:\:\times \prod_{\stackrel{[z]\in\Z_\ell /\Z}{\scriptscriptstyle{[z]\neq [0]}}}\Gstr .$$
Using the sets $\rset$, $\{\pm\}$ and the sandwich matrix $A$ from
normalised Rees matrix form $M[\Gstr;\rset,\{\pm\};A]$ for $\Sfib$ we can also express this as follows:
$\Ef(X_\theta)\backslash \{\Id\}$ is completely simple and topologically isomorphic to 
$$\Ef(X_\theta)\backslash \{\Id\} \cong M[\RSTf;I_\theta,\{\pm\};A],\quad \RSTf=\Gstr^{\Z_\ell/\Z}.$$
Here an entry $a_{\lambda,g}$ of $A$ is identified with the function $\tilde f\in \RSTf$ which satisfies $\tilde f(0)=a_{\lambda,g}$ and $\tilde f(z)=\one$ for regular $z$. 

Furthermore  $E(X_\theta)\backslash \Z$ is algebraically isomorphic to 
$$E(X_\theta)\backslash \Z \cong M[\RST;I_\theta,\{\pm\};A],\quad \RST = \RSTf\rtimes \Z_\ell$$
where $A$ is understood to take values in the subgroup $\Gstr\times  \{\one\}^{\Z_\ell/\Z\backslash \{[0]\}}\rtimes \{0\}$ of $\Gstr^{\Z_\ell/\Z}\rtimes \Z_\ell$. 
\end{thm}
\begin{proof}
As $\Ef(X_\theta)_0 =  \Sfib\cup \{\Id\}$ it is generated by its idempotents. 
Furthermore, the minimal rank $r_\pi$ is equal to $|\Aa|$. 
The first expression for $\Ef(X_\theta)$ is thus a direct application of Corollary~\ref{cor-main1}. 
The second expression corresponds to (\ref{eq-alter}). The last statement follows from Corollary~\ref{cor:reproducing M}. 
\end{proof}
While the isomorphism between $E(X_\theta)\backslash \Z$ and $ M[\Gstr^{\Z_\ell/\Z}\rtimes \Z_\ell;I_\theta,\{\pm\};A]$ is not continuous when $\Gstr^{\Z_\ell/\Z}\rtimes \Z_\ell$ is equipped with the product topology, Theorem \ref{thm-main2} makes clear where the non-tameness comes from: whereas the restrictions of $E^{fib}(X_\theta)$ to individual fibres are finite semigroups, it is the fact that the structure group of its kernel
consists of {\em all possible} functions over the orbit space which implies that 
the cardinality of $ E^{fib}(X_\theta)$, and hence of $E(X_\theta)$, is larger than that of the continuum.

\subsection{Height}\label{generalised-height}
The assumption of the last theorem, namely that $\Sfib$ is generated by its idempotents, is not always satisfied. To treat the other cases we introduce a new notion of height.

Let $\rset$ be the R-set of a bijective substitution $\theta$ and $\Gamma_\theta$ be the group generated by $\{gh^{-1}:g,h\in I_\theta\}$. We have seen in Lemma~\ref{lem-JSG} that $\Gamma_\theta$ is the Rees structure group of the subsemigroup generated by the idempotents of $\Sfib$. 
If $\rset$ contains $\id$ then  $\Gamma_\theta$ contains $\rset$ and therefore coincides with $\Gstr$, and consequently $\Sfib$ is generated by its idempotents. However, $\Gamma_\theta$ need not even be a normal subgroup of $\Gstr$; see Section~\ref{ex462} for an example.

\begin{lem} 
Let $\tilde\Gamma_\theta$ be a normal subgroup of  $\Gstr$ which contains the little structure group $\Gamma_\theta$. 
Denote by $\phi:\Gstr\to \Gstr/\tilde\Gamma_\theta$ the canonical projection. Then 
$\phi(g_1) = \phi(g_2)$ for any two elements of $\rset$.  In particular, $\Gstr/\tilde\Gamma_\theta$ is a finite cyclic group.
\end{lem} 
\begin{proof}
If $\phi(g_1)\neq \phi(g_2)$ for two elements of $\rset$ then $\phi(g_1 g_2^{-1})\neq 1$. But $g_1g_2^{-1}\in\Gamma_\theta\subseteq \ker\phi$. Since $\rset$ generates $\Gstr$,
the class of $\rset$ is a generator of $\Gstr/\tilde\Gamma_\theta$.
\end{proof}
We denote the order of $\Gstr/\tilde\Gamma_\theta$ by $\tilde h$.
Note that $\tilde h$ must devide any $\nu>0$ for which $(\theta^n)_\nu \in\tilde\Gamma_\theta$ (here $n$ is large enough such that $\nu\leq \ell^n$). Indeed, $(\theta^n)_\nu {(\theta^n)_0}^{-1}$ is a product of $\nu$ elements of $\rset$ and so its image in $\Gstr/\tilde\Gamma_\theta$ is $\nu$ times the generator of $\Gstr/\tilde\Gamma_\theta$.  In particular, $\tilde h$ devides $\ell-1$.

\subsubsection{Classical height}
To better understand the meaning of the quantity $\tilde h$ we recall the notion of height which occurs in the context of constant length substitutions.
\newcommand{\hc}{h_{cl}}

The equicontinuous factor $\Zl$ which we have described above for a primitive aperiodic substitution $\theta$ of constant length is not always the {\em maximal} equicontinuous factor, i.e.\ there might be an intermediate equicontinuous factor $(\Xmax, +g)$ between $(X_\theta,\sigma)$ and 
$(\Z_\ell, +1)$. The relevant quantity which governs this is the height of the substitution. In view of what follows we refer to it here as its {\em classical height}. 
This is a natural number arising as the height of a tower
comprising a  Kakutani-Rohlin model for the dynamical system and shows  up also in the spectral analysis. It can be computed as follows. Consider a 
one-sided fixed point $u=u_0u_1\cdots $ of $\theta$. 
The (classical) height $\hc$ of $\theta$   is defined as 
\begin{equation} \label{height}
\hc(\theta):= \max \{n\geq 1: \gcd(n,\ell)=1, n | \gcd\{k: u_k=u_0 \} \}\, .
\end{equation}
For primitive substitutions
it turns out to be independent of the choice of $u$. The following result was shown by Dekking \cite{dekking}, with partial 
results by Kamae \cite{kamae} and Martin \cite{martin}. 
\begin{thm} \label{thm:dekking}
Let $\theta$  be a primitive aperiodic substitution of  length $\ell$ and classical height $\hc$. Then 
the maximal equicontinuous factor of $(X_\theta, \sigma)$ is 
$(\Z_{\ell} \times \Z/\hc\Z , \add\times\add)$.
\end{thm}
The theorem says that $\frac1\hc$ corresponds to a topological eigenvalue of the dynamical system $(X_\theta,\sigma)$ which does not already occur in the spectrum of  $(\Z_l,+1)$. It moreover implies that the surjection $\tilde\pi$ in (\ref{eq-ext})
factors through $\Z/\hc \Z\times\Z_\ell$,
$$E(X_\theta)\twoheadrightarrow \Z/\hc \Z\times\Z_\ell  \twoheadrightarrow \Z_\ell $$ and therefore leads to a continuous surjective semigroup morphism 
 \begin{equation*} \label{classical grading} E^{fib}(X_\theta)\twoheadrightarrow \Z/\hc \Z  
\end{equation*}
Stated differently, $E^{fib}(X_\theta)$ is a $\Z/\hc \Z$-graded right topological semigroup.

A more detailed analysis yields the following picture.
Let $u$ be any one-sided fixed point of $\theta$.
For $a\in \mathcal A$, let $i_u(a) = \min\{k:u_k=a\}$. We claim that the set $\{n\in\N: u_n=a\}$ of occurences of $a$ in $u$ is contained in $i_u(a)+\hc\N$ where $\hc$ is as in \eqref{height}. For, say that $a$ occurs at indices $i$ and $j$ in $u$. Let $v$ be the one-sided fixed point of $\theta$ that starts with $a$. By minimality there exists $i_0\in \N_0$ such that we see $a$ in $v$ at the indices $i_0+i$ and $i_0+j$, $a=v_{i_0+i}=v_{i_0+j}$.  
Recall that the height $h$ can be defined using any fixed point of $\theta$. Taking $v$ in place of $u$ in Definition \eqref{height} we see that 
all indices at which we see the letter $a$ in $v$ are multiples of $h$. Thus $h$ divides $i-j$, and our claim follows.
In other words, we can partition the alphabet into subsets $\mathcal A_k:=\{a\in \mathcal A: i_u(a)\equiv k\bmod h \}$ and $\sigma(A_k) = A_{k+1}$.
Note also that if $\theta$ is simplified then $\{k: u_k=u_0 \}$ contains $\ell-1$ and hence the height must divide $\ell-1$. {In Lemma \ref{general-height-vs-height}  and Theorem \ref{thm-grading}, we extend this partition to $\Gstr$ and $E(X_\theta)$.}

\begin{lem}\label{general-height-vs-height}
Let $\theta$ be a primitive aperiodic bijective substitution with structure group $\Gstr$. If $\theta$ has classical height $h_{cl}$, then there is a surjective group homomorphism $\phi_{cl}:\Gstr\to \Z/h_{cl}\Z$ such that  for all $g\in \rset$ 
we have $\phi_{cl}(g) = 1$.
\end{lem}
\begin{proof} We may assume that $\theta$ is simplified and hence the height $h_{cl}$ divides $\ell-1$.
Fix an arbitrary one-sided fixed point $u=u_0u_1\cdots$ of $\theta$. 
For $a\in \mathcal A$, let $i_u(a) = \min\{k:u_k=a\}$; we have seen that  $\{n\in\N: u_n=a\}$ is contained in $i_u(a)+\hc\N$.
We now understand $k$ and $i_u(a)$ as an index modulo $\hc$. As
$\theta_j(u_k) = u_{\ell k+j}$ we see that $i_u(\theta_j(a)) - i_u(a) \equiv (\ell-1) i_u(a) + j $. Since the height must divide $\ell-1$ we find  $(\ell-1) i_u(a) + j  \equiv j$. Hence $i_u(\theta_j(a)) - i_u(a)$ does not depend on $a$ and so 
$\phi_{cl}(\theta_j) := i_u(\theta_j(a)) - i_u(a)$ is a well defined map from $\{\theta_j,j\geq 0\}$ to $\Z/\hc\Z$. We compute $i_u(\theta_{j'}\theta_j(a)) \equiv j'+\ell(j+\ell i_u(a)) \equiv j'+j+i_u(a)$ and thus see that $\phi_{cl}$ is multiplicative. It hence  
induces a surjective group homomorphism $\phi_{cl} :\Gstr \to \Z/h_{cl}\Z$. Clearly $\phi_{cl}(\theta_j\theta_{j-1}^{-1}) = 1$.
\end{proof}

\subsubsection{Generalised height}\label{sec:generalised height}
To discuss generalised height we introduce the maps $\evo$ and $\evo^z$, where
$$\evo:X_\theta\to\Aa,\quad \evo(x) := x_0$$
reads the letter on $0$ of $x$ and 
$\evo^z:= \evo|_{e\pi^{-1}(z)}$ is the restriction of $\evo$ to the subset $e\pi^{-1}(z)$ of the fibre over $z$. Here again, $e$ is a chosen minimal idempotent, but we take it from $E(X_\theta,\Z^+)$. Clearly, if $z$ is regular then $e\pi^{-1}(z)=\pi^{-1}(z)$. Note that if $z\in \Z^+$ then $\evo(x) = \evo(e(x))$ because $e$ does not affect the right infinite part of a fixed point. $\evo^z$ 
is a bijection which we will use below to identify $e\pi^{-1}(z)$ with $\Aa$. As we already mentioned, 
conjugation with $\evo^0$ identifies $G_\pi$ with $G_\theta$ and $\Gamma_\pi$ with $\Gamma_\theta$. 

\begin{lem} \label{lem-fz}
Let $f\in E(X_\theta,\Z^+)$ and $z\in\Z_\ell\backslash \Z^-$. 
There exists $f_z\in \Gstr$ such that for all $x\in \pi^{-1}(z)$ we have $\evo(f(x)) = f_z(x_0)$.
\end{lem}
\begin{proof}
We first show the result for $z=n\in \Z^+$ and $f=\sigma^m$, $m\geq 0$. Then $x = \sigma^n(a.b)$ for some fixed point $a.b$ and $n\geq 0$. It follows that $\sigma^m(x) = (\theta^{k})_{n+m} (\theta^{k})_{n}^{-1}(x_0)$ for $x\in \pi^{-1}(z)$ and all $k$ with 
$\ell^k>n+m$. Thus for $z=n$ we can take ${\sigma^m}_z = (\theta^{k})_{n+m} (\theta^{k})_{n}^{-1}$. 

Next suppose $z$ is regular and $f=\sigma^m$.
Since $(X,\sigma)$ is forward minimal there is $h\in E(X_\theta,\Z^+)$ with $\tilde\pi(h)=z$, where
$h$ is the limit of some generalised  sequence $\sigma^{n_\nu}$ and $\pi^{-1}(z) = h(\pi^{-1}(0))$. 
Let $x\in \pi^{-1}(z)$, i.e.\ $x=h(a.b)$ for some fixed point $a.b$. By continuity of $\sigma$ and $\evo$ we have
$$\evo\circ \sigma^m\circ h (a.b) = \lim_{\nu} \evo\circ \sigma^{m+n_\nu}(a.b)$$ and since  $\pi^{-1}(0)$ is finite there exists $\nu_0$ such that for all $\nu\geq \nu_0$ and all $a.b\in\pi^{-1}(z)$, 
$\lim_{\nu'} \evo\circ \sigma^{m+n_{\nu'}}(a.b) =  \evo\circ \sigma^{m+n_\nu}(a.b)$. 
Hence
$${\sigma^m}_z =  {\sigma^{m}}_{n_\nu}$$
once $\nu\geq \nu_0$.

Now let $f\in E(X_\theta,\Z^+)$ and $z\in\Z_\ell\backslash \Z^-$. $f$ is the limit of some generalised  sequence $\sigma^{n_\mu}$. We have
$$\evo\circ f (x) = \lim_{\mu} \evo\circ \sigma^{n_\mu}(x)$$
As $\pi^{-1}(z)$ is finite there exists $\mu_0$ such that for all $\mu\geq \mu_0$ and all $x\in \pi^{-1}(z)$, 
$\lim_{\mu'} \evo\circ \sigma^{n_{\mu'}}(x) =  \evo\circ \sigma^{n_\mu}(x)$. 
Hence $f_z = {\sigma^{n_\mu}}_z$ once $\mu\geq \mu_0$.
\end{proof}

\begin{thm}\label{thm-grading} Let $\tilde\Gamma_\theta$ be a normal subgroup of  $\Gstr$ which contains the little structure group $\Gamma_\theta$. 
There exists a continuous surjective semigroup morphism
$$\eta: E(X_\theta,\Z^+)\to \Gstr/\tilde\Gamma_\theta\cong  \Z/\tilde h\Z$$
such $\eta(\sigma f) = \eta(f)+1$. In other words, $\frac1{\tilde h}$ is a topological eigenvalue of the dynamical system $(E(X_\theta,\Z^+),\lambda_\sigma)$.
\end{thm}
\begin{proof}
By Lemma~\ref{lem-fz}, given $z\in\Z_{\ell}\backslash\Z^-$ we can assign to any $f\in E(X_\theta,\Z^+)$ a map $f_z\in I_\theta$. It depends on $z$ but since  two elements of $\rset$ differ by an element of $\Gamma_\theta$ the class $f_z \tilde\Gamma_\theta$ is independent of $z$.
We can therefore define $$\eta(f) := f_z\tilde\Gamma_\theta.$$ 
We show that $\eta$ is locally constant. 
For that we consider the following neighbourhood of $f$:
given $z\notin\Z^-$ let $$V = \{f'\in E(X_\theta,\Z^+): \evo(f(x)) = \evo(f'(x))\, \forall x\in\pi^{-1}(z) \}.$$ Then, if $f'\in V$ we must have $f_z = f'_{z}$ and hence $\eta(f) =\eta(f')$. Hence $\eta$ is locally constant, so continuous. It is immediate that $\eta(\sigma f) = \eta(f) + 1$.
By continuity we obtain $\eta(f_1 f_2) = \eta(f_1) + \eta(f_2)$.

Finally $e^{2\pi i\eta}$ is a continuous eigenfunction for the eigenvalue $\frac1{\tilde h}$.
\end{proof}
Let us compute the degree of the map $f_{[L\cdot R,+]}$ where $L \cdot R = (\theta^n)_{\nu-1}\cdot (\theta^n)_{\nu}$, $\nu\geq 0$. 
We have seen that $f_{[(\theta^n)_{\nu-1}\cdot (\theta^n)_{\nu},+]}$ is the limit of a subnet of $(\sigma^{\nu \ell^k})_k$. Furthermore $\tilde h$ devides $\ell-1$. Hence $\nu$ and $\nu\ell^k$ agree modulo $\tilde h$ showing that 
$\eta({\sigma^{\nu \ell^k}})$ is independent of $k$ and equal to $\nu$ modulo $\tilde h$. Thus
$$\eta(f_{[(\theta^n)_{\nu-1}\cdot (\theta^n)_{\nu},+]}) =  \nu\: \mbox{modulo}\: \tilde h.$$ 
The above calculation shows also that the degree of an element of 
$\M^{fib}(X_\theta,\Z^+)$ depends only on its restriction to the singular fibre $\pi^{-1}(0)$.
It is hence determined by a grading on $\M^{fib}_0(X_\theta,\Z^+)\cong \rset \times \Gstr\times\{+\}$. The latter can be easily obtained using the bijection (\ref{eq-bij}) between
$G^{(2)}_+$ and $\rset \times \Gstr\times\{+\}$; it is given by 
 $\eta:\rset \times \Gstr\times\{+\} \to \Gstr/\tilde \Gamma_\theta $,
$$ \eta(i,g,+) = g \tilde \Gamma_\theta.$$
In particular, as it should be, idempotents have degree $0$. 
\bigskip

As an aside we remark that the grading $\eta$ can be extended to all of $E(X_\theta)$. We can repeat the above with $E(X_\theta,\Z^-)$ while exchanging $\sigma$ with $\sigma^{-1}$ and $\rset$ with $\rset^{-1}$. One then finds that 
$\eta(f_{[(\theta^n)_{\nu-1}\cdot (\theta^n)_{\nu},-]}) = \nu -1$ modulo $\tilde h$ and 
$\eta(i,g,-) = g \tilde \Gamma_\theta$.
\bigskip

The grading of $E(X_\theta,\Z^+)$ restricts to a grading of the group $\RST = eE(X_\theta,\Z^+) e$.
We denote the elements of degree $k$ by ${\RST}_k := \eta^{-1}(\{k\})\cap \RST$. Thus $ \RST  = \bigsqcup_{k\in \Z/\tilde h\Z} {\RST}_k$ with ${\RST}_k {\RST}_l = {\RST}_{k+l}$ and
${\RST}_k  = \fg^k {\RST}_0$
for any choice of element $\fg\in {\RST}_1$.
Similarly, the grading restricts to $\RSTf$, the structure group of the kernel of $\Ef$.

Recall the definition \eqref{eq:def T} of $\Et$ as the subsemigroup of elements of $\Ef$ which act trivially on all regular fibres. Since we have only one orbit of singular fibres, its restriction $\Et_0$ to the fibre at $0\in\Z_\ell$ is faithful and  $e\Et e$ isomorphic to $e\Et_0 e$. 
We identify  $e\pi^{-1}(z)$ with $\Aa$ through the restricted evaluation map $\evo^z:e\pi^{-1}(z) \to \Aa$. Define $\Et_\theta:= \evo^0 e\Et_0 e (\evo)^{-1}$, a subgroup of the group of bijections of $\Aa$.
We now consider the situation in which $\tilde\Gamma_\theta$ is a subgroup of $\Et_\theta=\evo^0 e\Et_0 e (\evo)^{-1}$
and define
\begin{equation*}\label{def: CfGamma}\Cf(\tilde\Gamma_\theta) :=\{\tilde f\in e\,\Cf\, e\, : (\Phi^z_{0})^{-1}(\tilde f(z))\in (\evo^0)^{-1} \tilde\Gamma_\theta \evo^0 \:\forall z\in\Z_{\ell}\} . \end{equation*}
This group is independent of the choice of lift to define $\Phi^z_{0}$, as $\tilde\Gamma_\theta$ is normal in $\Gstr$.
\begin{prop}\label{fibre-structure-height1} Let $\tilde\Gamma_\theta$ be a normal subgroup of $\Gstr$ which contains $\Gamma_\theta$ and is contained in $\Et_\theta$. Then, w.r.t.\ the grading defined by $\tilde\Gamma_\theta$, we have the inclusion of groups
$$\RSTf_0  \subset \Cf(\tilde\Gamma_\theta) . $$
\end{prop}

\begin{proof} 
Given $f\in \RSTf$ we have $\eta(f)  =0$ if and only if, for all $z\in\Z_\ell\backslash \Z^-$ we have ${f}_z\in \tilde\Gamma_\theta$. Note that,  with $\tilde f$  defined as in \eqref{eq:definition-f-tilde}, $f_z\circ \evo^z = \evo^z\circ \tilde f(z)$. 
Moreover, if $h\in E(\Z^+)$ then $h_z\circ \evo^z = \evo^{z+\pi_*(h)}\circ h\left|_{\pi^{-1}(z)}\right.$.
It follows from Lemma~\ref{lem-fz} that $\evo^{z+\pi_*(h)}\circ h\left|_{\pi^{-1}(z)}\right.\circ (\evo^z)^{-1}\in G_\theta$. We apply this to   $h=s(z)$, {the lift of $z$}, to see that 
$g:=\evo^{z}\circ s(z)\left|_{\pi^{-1}(0)}\right.(\evo^0)^{-1}\in G_\theta$. Now
\begin{align*}(\Phi^z_{0})^{-1}(\tilde f(z)) &= s(z)^{-1} \tilde f(z) s(z) = s(z)^{-1} \circ(\evo^z)^{-1}\circ f_z\circ  \evo^z\circ  s(z)\\& \in (\evo^0)^{-1}\circ g^{-1}\tilde\Gamma_\theta g \circ \evo^0 =  (\evo^0)^{-1}\circ \tilde\Gamma_\theta  \circ \evo^0\end{align*}
as $\tilde \Gamma_\theta$ is normal in $G_\theta$. Thus 
$f\in   \Cf(\tilde\Gamma_\theta) $. \end{proof}
\begin{definition} The normal completion $\Glstr$ of the little structure group $\Gamma_\theta$ is the smallest normal subgroup of $\Gstr$ which contains $\Gamma_\theta$.
The generalised height $h$ of a primitive aperiodic bijective substitution is the  order of 
$\Gstr/\Glstr$.
\end{definition}
\begin{prop} Generalised height must be at least as large as classical height.
\end{prop}
\begin{proof}
Recall the quotient map $\phi_{cl}:\Gstr\to \Z/h_{cl}\Z$ from Lemma~\ref{general-height-vs-height}.
It satisfies $\phi_{cl}(\Gamma_\theta) = 0$. As  $\Glstr$ is generated by elements of the form $g h g^{-1}$ with $h\in \Gamma_\theta$ and $g\in \Gstr$ we have $\phi_{cl}(\Glstr) = 0$
and hence $\Gstr/\Glstr$ factors onto $\Z/h_{cl}\Z$.
\end{proof}
In Section~\ref{ex462} we provide an example with trivial classical height but non-trivial generalised height.
\begin{prop}\label{fibre-structure-height2} 
$\Glstr= \Et_\theta$ and
$$\RSTf_0  =  \Cf(\Glstr)\cong \Glstr^{\Z_\ell/\Z} .$$
\end{prop}
\begin{proof} By Theorem~\ref{thm-main} we have  the inclusion of groups $ \Cf(\Et_\theta) = e\Ct e \subseteq e\Ef e=\RSTf$. Hence, by Proposition ~\ref{fibre-structure-height1}, 
$ \Cf(\Et_\theta)$ is a subgroup of $\bigsqcup_{k\in \Z/\tilde h\Z} \mathfrak f^k \Cf(\Glstr)$ 
where $\mathfrak f$ is some element from $\RSTf_1$. 
Since $s(z)$ and $s(z)^{-1}$ have opposite degree, $(\Phi_0^z)^{-1}(\tilde{\mathfrak f}(z)) \Glstr\in \Gstr/\Glstr$ must be the generator of $\Gstr/\Glstr$ for all $z$. Hence, for any $f\in \Cf(\Et_\theta)$, $(\Phi_0^z)^{-1}(\tilde{f}(z)) \Glstr\in \Gstr/\bar\Gamma_\theta$ is constant in $z$. 
This is possible only if $\Et_\theta$ is a subgroup of $\Glstr$. 
If that is the case then $\Cf(\Et_\theta) = \RSTf_0$, again by Proposition~\ref{fibre-structure-height1}. 

The covariance condition says that a function $\tilde f\in \Cf$ is determined on the $\Z$-orbit of a point $z\in\Z_\ell$ by its value on $z$. We may hence chose for each $\Z$-orbit $[z]\in\Z_\ell/\Z$ a representative $z$ and then the obtain a bijection from  $\Cf(\Glstr)$ to $\Glstr$-valued functions over the orbit space $\Z_\ell/\Z$ by restricting $\tilde f$ to the chosen representatives. This bijection is, of course, not canonical as it involves an uncountable choice of representatives, but it is a topological isomorphism of semigroups. 
\end{proof}

We end this section with the generalisation of Theorem~\ref{thm-main2} to substitutions which may have non-trivial height.
 \begin{thm}\label{thm-main4}
Let $\theta$ be a primitive aperiodic bijective substitution with generalised height $h$.
 Using the sets $\rset$, $\{\pm\}$ and the sandwich matrix $A$ from
normalised Rees matrix form $M[\Gstr;\rset,\{\pm\};A]$ for $\Sfib$, 
$\Ef(X_\theta)\backslash \{\Id\}$ is 
topologically isomorphic to 
$$\Ef(X_\theta)\backslash \{\Id\} \cong M[\RSTf;I_\theta,\{\pm\};A].$$
where $\RSTf$ is a $\Z/h\Z$-graded group and
$$\RSTf_k  =  \fg^k \Cf(\Glstr) $$
for any  element $\fg$ of degree $1$ of $\RSTf$. 
Here an entry $a_{\lambda,g}$ of $A$ is identified with the function $\tilde f\in \RSTf_0$ which satisfies $\tilde f(0)=a_{\lambda,g}$ and $\tilde f(z)=\one$ for regular $z$. 

If $\Gstr$ contains an element of order $h$ then    $\RSTf$ is a semidirect product
$$\RSTf \cong \Cf(\Glstr)\rtimes \Z/h\Z.$$

Furthermore  $E(X_\theta)\backslash \Z$ is algebraically isomorphic to 
$$E(X_\theta)\backslash \Z \cong M[\RST;I_\theta,\{\pm\};A]$$
where $\RST$ is the extension determined by $\RSTf\hookrightarrow \RST \twoheadrightarrow \Z_\ell$ and $A$ is understood to take values in the subgroup $\RSTf$. 
If moreover, generalised height is equal to classical height then 
there is a split section $s:\Z_\ell \to \RST$ whose image belongs to ${\RST}_0$. In particular we have the algebraic isomorphism
$$\RST\cong \RSTf\rtimes \Z_{\ell}.$$
 \end{thm}
\begin{proof}
The first part follows from Prop.~\ref{fibre-structure-height2}. 
 It remains to show that $\RSTf \cong \Cf(\Glstr)\rtimes \Z/h\Z$ provided 
$\Gstr$ contains an element of order $h$. An element $\fg\in \RSTf$ with $\eta(\fg) = 1$ can be constructed as follows. We saw that any $(i_0,i,+)\in M[G_\theta;\rset,\{\pm\},A]$ with $i\in \rset$ has degree $1$. Pick $i\in \Gstr$ and
let $\fg$ be the function which satisfies 
$$ \fg(z):=\Phi_0^z((\evo^0)^{-1}(i_0,i,+)\evo^0). $$
Then $\fg$ is an element of degree $1$ in $\Cf(\Gstr)$. 
As  $\Phi_0^z$ restricts to a group isomorphism from $e\Ef_0e$ to $\Ef_z$ it preserves the order of an element and therefore, if $i^h=\one$, then $\fg$ has order $h$.  In that case, $\Z/h\Z \ni 1 \mapsto \fg \in \RSTf$ induces a split section for the exact sequence
$\RSTf_0\hookrightarrow \RSTf\stackrel{\eta}\twoheadrightarrow \Z/h\Z$.

The result for $E(X_\theta)$ follows from Proposition~\ref{extension-semigroup}. 
We have seen that, if $h$ is equal to the classical height, then $\tilde\pi$ factors through the maximal equicontinuous factor which is
$\Z/h\Z\times\Z_\ell$. By the axiom of choice we can construct a covariant lift 
$s:\Z_\ell \to \RST$ which factors through $\{0\}\times\Z_\ell\subset \Z/h\Z\times\Z_\ell$. It hence takes values in   ${\RST}_0$. It follows that the split homomorphism $\hat s$ constructed in Proposition~\ref{prop-shomo} also takes values in ${\RST}_0$. Thus 
${\RST}_0 \cong \RSTf_0\rtimes \Z_\ell$. This implies the last statement.
\end{proof}

\section{The virtual automorphism group of unique singular orbit systems}\label{Ellis-group}

In this section we investigate the relationship between the automorphism group and the virtual automorphism group of bijective substitutional systems, as defined by 
Auslander and Glasner \cite{AG-2019}. They show that an almost automorphic system is semi-regular iff it is equicontinuous. They also show that the Thue-Morse shift is semi-regular. Using our tools, in this section we  will extend their result to show that bijective substitution shifts are semi-regular. 

We start a little bit more generally considering minimal systems $(X,T, \sigma)$ where $T$ 
is an abelian group, unless we are talking about a substitution, in which case $T=\Z$.

\subsection{Automorphism groups of bijective substitutions}

The {\em automorphism group} $\Aut(X)$ of a dynamical system $(X,\sigma)$ is the group, under composition, of all homeomorphisms of $X$ which commute with $\sigma$. 
Since the elements of $E(X)$ are limits of generalised sequences of powers of $\sigma$, the automorphism group, viewed as a subset of $X^X$, lies in the commutant of $E(X)$. 
Similar to the situation of Ellis semigroups described in Section~\ref{semigroup-of-factor},
if $(X,\sigma)$ is minimal and $\pi: X \rightarrow Y$ is an equicontinuous factor map, 
then the map $\pi_*$ from Section~\ref{semigroup-of-factor} is well-defined on automorphisms of $X$ 
 inducing a group morphism $\pi_*: \Aut(X)\rightarrow \Aut(Y)$; see \cite{coven-quas-yassawi}. We also have $\Aut(Y)\cong Y$ for a minimal equicontinuous system. But $\pi_*$ is usually not surjective, although its image always contains $T\subset Y$.  
We can therefore analogously define $\Aut^{fib}(X)$ as the kernel of $\pi_*$, that is, the subgroup of automorphisms which preserve the $\pi$-fibres, and then determine $\Aut(X)$ through the extension $\Aut^{fib}(X) \hookrightarrow \Aut(X) \twoheadrightarrow \pi_*(\Aut(Y))$.  Contrary to most elements of $E^{fib}(X)$, the elements of $\Aut^{fib}(X)$ are always continuous and therefore, again by minimality, $\Aut^{fib}(X)$ is determined by its restriction to a fibre $\pi^{-1}(y_0)$. Furthermore, $\Aut^{fib}(X)$ must commute with $\RSTfp=eE^{fib}(X)e$ and therefore its restriction to $\pi^{-1}(y_0)$ is
 contained in the centraliser of the structure group $G_\pi=eE^{fib}_{y_0}(X)e$ in the permutation group $S_{e\pi^{-1}(y_0)}$ of the factor
 \cite[Theorem 33]{M-Y}. For primitive bijective substitution shifts it can be shown that $\Aut^{fib}(X)$ exhausts that centraliser and that 
$\pi_*: \Aut(X)\rightarrow \Aut(Y)$ is as small as possible: 

\begin{thm} \cite[Theorem 5]{L-M}\label{automorphism-theorem}
Let $\theta$ be a primitive aperiodic bijective substitution over the alphabet $\Aa$. Then $\Aut^{fib}(X_\theta)$ is isomorphic to the centraliser $C_{S_{\mathcal A}} (G_\theta) $ and
$$\Aut(X_\theta)\cong C_{S_{\mathcal A}} (G_\theta) \times \Z.$$
\end{thm}

\subsection{Virtual automorphism groups}

We start with the two algebraic lemmas. Given a set  $X$, $x\in X$ and a group $G\subseteq S_X$, the permutation group of $X$,  let $\Stab_G(x)$ denote the stabilizer of $x$ in $G$, let $N_G(\Stab_G(x))$ 
denote the normaliser of $\Stab_G(x)$ in $G$, and let $C_{S_X}(G)$ denote the centraliser of $G$ in $S_X$.
The following lemma from group theory  is well-known but we include a proof.
Recall  that a group G acts {\em transitively} on $X$ if for each pair $x, y$ in $X$ there exists $g\in G$ such that $g(x) = y$.

\begin{lem}\label{algebraic-lemma}
Let $G$ be a subgroup of the permutation group $S_X$ which acts transitively on $X$, and let $x\in X$. Then $$N_G(\Stab_G(x)) / \Stab_G(x) \cong C_{S_X}(G).$$
\end{lem}
\begin{proof}
Let $h\in N_G(\Stab_G(x))$.  We first claim that if two $f, f'\in G$ satisfy $f(x)=f'(x)$,  then $fh^{-1}(x)= f'h^{-1}(x)$. Indeed, $f(x)=f'(x)$ implies  
$f' = f s$ for some  $s\in \Stab_G(x)$. Furthermore $sh^{-1} = h^{-1}s'$ with $s'$ in $\Stab_G(x)$, as $h$ normalises $\Stab_G(x)$. Hence 
\[fh^{-1}(x) = fh^{-1}(s'x) =   f s h^{-1}(x)=   f'h^{-1}(x).\]
Therefore the map 
$ F: N_G(\Stab_G(x)) / \Stab_G(x) \to S_X$ defined by
\begin{equation*} F(h)(y) := fh^{-1}(x),\quad \mbox{for any $f\in G$ such that $f(x)=y$} \label{definition-phi} \end{equation*}
is well-defined. We show $F$ preserves the group multiplication. Let $h_1,h_2\in  N_G(\Stab_G(x)) /\Stab_G(x)$. Then 
\begin{equation}\label{eq-F-homo}
F(h_1)(F(h_2)(y)) = f_1h_1^{-1}(x) ,\quad \mbox{for any $f_1\in G$ such that 
$f_1(x)=F(h_2)(y)$}
\end{equation}
and $F(h_2)(y) = f_2h_2^{-1}(x)$ for any $f_2\in G$ such that $f_2(x) = y$. In particular,  $f_1(x)=f_2h_2^{-1}(x)$ so that we may take $f_1=f_2h_2^{-1}$ in (\ref{eq-F-homo}).
Thus $F(h_1)(F(h_2)(y))=f_2h_2^{-1}h_1^{-1}(x)$. As  $f_2(x) = y$ this is also equal to
$F(h_1h_2)(y)$.

Note that if  $F(h) = \Id$ then $fh^{-1}(x) = f(x)$ for all $f$, which implies that $F$ is injective.

Furthermore, using our first claim again, it can be checked that $F(h)$ commutes with the elements of $\Stab_G(x)$ and so the image of $F$ is contained in 
$C_{S_X}(G)$.
To see that the image of $F$ is all of $C_{S_X}(G)$ let $\psi\in C_{S_X}(G)$. Since $G$ acts  transitively there exists $h\in G$ such that $\psi(x) = h^{-1}(x)$. Then, for $f\in G$, 
\begin{equation}\psi(f(x)) = f\psi(x) = fh^{-1}(x).\label{surjective}\end{equation}
If we apply this formula to $f\in \Stab_G(x)$ we get $h^{-1}(x) = \psi (x) = \psi(f(x)) =   	fh^{-1}(x)$ which implies that $h^{-1}$ normalises $\Stab_G(x)$. Now \eqref{surjective} shows that  $\psi=F([h])$ where $[h]$ is the class of $h$ in  $N_G(\Stab_G(x)) /\Stab_G(x)$.
\end{proof}

Let $G_1$, $G_2$ be two groups acting on $X_1$ and $X_2$, resp..
Let $G\subset G_1\times G_2$ a subgroup such that the projection $p_1:G\to G_1$ is surjective. $G_1\times G_2$ acts on $X_1$ via the projection $p_1$ hence  $G$ also acts on $X_1$: if $f\in G$ then $f(x):= f_1(x)$ where $f_1=p_1(f)$. 
\begin{lem} \label{lem-exercice} 
In the above context, let $x\in X_1$. Then $p_1$ induces an isomorphism
$$N_G(\Stab_G(x))/\Stab_G(x)  \to N_{G_1}(\Stab_{G_1}(x))/\Stab_{G_1}(x).$$
\end{lem} 
\begin{proof}
$f\in G$ belongs to $\Stab_G(x)$ iff $f_1(x) = x$. Hence 
$p_1^{-1}(\Stab_{G_1}(x))=\Stab_G(x)$. In particular $p_1$ descends to the quotients.
$h\in G$ belongs to $N_G(\Stab_G(x))$ iff $\forall f\in \Stab_G(x)$ we have $fh^{-1}(x) = h^{-1}(x)$. Since $p_1^{-1}(\Stab_{G_1}(x))=\Stab_G(x)$ the latter is equivalent to
$\forall f_1\in \Stab_{G_1}(x)$ we have $f_1h_1^{-1}(x) = h_1^{-1}(x)$. Hence
$N_{G_1}(\Stab_{G_1}(x))=p_1(N_G(\Stab_G(x)))$. Hence the map induced by $p_1$ is surjective. 
$p_1^{-1}(\Stab_{G_1}(x))=\Stab_G(x)$ shows that $p_1(h) \in \Stab_{G_1}(x)$ implies
$h\in \Stab_G(x)$ and so the map induced by $p_1$ is injective.
\end{proof}

\subsubsection{Definition of the virtual automorphism group}

Let $(X,\sigma,T)$ be a minimal system with abelian group $T$. We denote its Ellis semigroup here simply by $E$. 
Let $e\in E$ be a minimal idempotent and 
recall that  $\RSTp =eEe$ is a group. Pick $x_0\in e(X)$. Applying Lemma \ref{algebraic-lemma} to the space $e(X)$ and the group $G=\RSTp$, we get
\[  N_{\RSTp  } (\Stab_{\RSTp  }(x_0)   ) / (\Stab_{\RSTp  }(x_0)) \cong C_{S_{e(X)}}(\RSTp  ).     \] 
 
 \begin{definition}\label{vag-definition}
 Let $(X,\sigma,T)$ be minimal and $e\in E$ be a minimal idempotent in its Ellis semigroup.
  The {\em virtual automorphism group $V(X)$ of $(X,\sigma)$} is defined to be  $C_{S_{e(X)}}(\RSTp  )$.
 \end{definition}

While this definition depends on the choice of the minimal idempotent $e$ it does so only up to isomorphism, as different choices of idempotents lead to isomorphic groups.

We remark that the restriction map 
$$\Aut(X)\ni \Phi\mapsto\Phi|_{e(X)}\in C_{S_{e(X)}}(\RSTp  )$$ is well-defined, as automorphisms commute with the elements of $E(X)$. Furthermore, if $e(X)$ is dense in $X$ then this restriction map is an injective group homomorphism, as automorphisms are
continuous. A condition guaranteeing that $e(X)$ is dense in $X$ is point distality of 
$(X, \sigma,T)$, as minimal idempotents fix distal points. So for a point distal minimal system $(X, \sigma,T)$, $\Aut(X)$ is a subgroup of $V(X)$.
 
 \begin{definition}\label{semi-regular-definition}
 The minimal dynamical system  $(X,\sigma,T)$  is called {\em semi-regular} 
 if the restriction map $\Phi\mapsto\Phi|_{e(X)}$ is an isomorphism between  
 the automorphism group $\Aut(X)$ and the virtual automorphism group $V(X)$.  \end{definition}
 
Note that our definition is slightly different to Auslander and Glasner's, who simply require that the map $\Aut(X)\rightarrow V(X)$ be onto. However for point distal systems, and the systems we study here are point distal, the definitions coincide.

\subsubsection{Unique singular fibre systems} 
We investigate the virtual automorphism group $V(X)$ for minimal systems which have an equicontinuous  factor with a unique orbit of singular points.  
Recall the definition (\ref{eq:def T}) of $\Et$ as the subsemigroup of elements of $\Ef$ which act trivially on regular fibres and that it restriction $\Et_{y_0} $ to $\pi^{-1}(y_0)$ is faithful. 
Recall the notation $\RSTfp=e\Ef e$.
\begin{lem} \label{lem-aus1}
Let $(X,\sigma,T)$ be a minimal unique singular orbit system. 
Let $y_0$ be a singular point in its maximal equicontinuous factor and $x_0\in e\pi^{-1}(y_0)$.
If $e\Et_{y_0} e$ acts effectively\footnote{For any $x\in  e\pi^{-1}(y_0)$ there exists $\gamma\in e\Et_{y_0} e$ such that $\gamma(x)\neq x$.}
on $e\pi^{-1}(y_0)$, then
$$N_{\RSTp  }(\Stab_{\RSTp  }(x_0)) = N_{\RSTfp  }(\Stab_{\RSTfp  }(x_0)) \times T.$$
\end{lem}
\begin{proof} 
By Theorem~\ref{thm-main} $\Ef$ contains $\Ct$ as defined in (\ref{eq:def Ct}) and therefore 
\begin{equation}\label{eq-aus1}
\Stab_{\RSTp  }(x_0)\supset \{\tilde f\in  e\Ct e  :\tilde f (y_0)(x_0)=x_0\}.
\end{equation}
Let $h\in N_{\RSTp  }(\Stab_{\RSTp  }(x_0))$. Then $f (h(x_0))=h(x_0)$ for all $f\in \Stab_{\RSTp  }(x_0)$. Using $y:=\tilde\pi(h)+y_0$ this means that $\tilde f(y)(h(x_0)) = h(x_0)$. Suppose that $y-y_0 \notin T\subset Y$ ($T$ seen as a subgroup of the maximal equicontinuous factor $Y$). 
Since $e\Et_{y_0} e$ acts effectively there is $\gamma\in e\Et_{y_0} e$ such that 
$\Phi^{y}_{y_0}(\gamma)(h(x_0)) \neq h(x_0)$. 
By (\ref{eq-aus1}) there exists $f\in \Stab_{\RSTp  }(x_0)$ such that $\tilde f(y) = \Phi^{y}_{y_0}(\gamma)$. For this element we have 
$f(h(x_0)) =\Phi^{y}_{y_0}(\gamma)(h(x_0))\neq h(x_0)$. This is a contradiction and thus all  $h\in N_{\RSTp  }(\Stab_{\RSTp  }(x_0))$ must satisfy $\tilde\pi(h)\in T$. Clearly $\sigma\in N_{\RSTp  }(\Stab_{\RSTp  }(x_0))$ and so the above shows that  
the exact sequence $ {\RSTfp  }\hookrightarrow {\RSTp  }\twoheadrightarrow Y $
restricts to the exact sequence
$$ N_{\RSTfp  }(\Stab_{\RSTfp  }(x_0))\hookrightarrow N_{\RSTp  }(\Stab_{\RSTp  }(x_0)) \twoheadrightarrow T .$$
We can lift the subgroup $T\subseteq Y$ with the lift $s:T\to E$ given by  
$s(t) = \sigma^t$ to see that $N_{\RSTp  }(\Stab_{\RSTp  }(x_0))$ is a semi-direct product, which is in fact direct, as $\sigma^t$ commutes with $E$, since $T$ is abelian.
\end{proof}

\begin{lem} \label{lem-vag-structure}
Let $(X,\sigma,T)$ be a minimal system with equicontinuous factor $\pi:X\to Y$.
Let $y_0\in Y$, $x_0\in e\pi^{-1}(y_0)$ where $e$ is a minimal idempotent of $E(X)$. 
We have
$$N_{\RSTfp  }(\Stab_{\RSTfp  }(x_0))\cong C_{S_{e\pi^{-1}(y_0)}}(G_\pi) .$$
\end{lem}
\begin{proof}
$\RSTfp  $ is a subgroup of $e\Cf e=\prod_{y\in Y_0} e\Ef_y e$ where $Y_0\subset Y$ contains exactly one representative for each orbit and we suppose that $y_0\in Y_0$. Thus $e\Cf e$ has the form $G_1\times G_2$ with $G_1 = G_\pi = e\Ef_{y_0} e$ and the projection $p_1$ is surjective. We thus can apply Lemma~\ref{lem-exercice} and then Lemma~\ref{algebraic-lemma} to see that 
$$N_{\RSTfp  }(\Stab_{\RSTfp  }(x_0))\cong N_{e\Ef_{y_0} e}(\Stab_{e\Ef_{y_0} e}(x_0))\cong  C_{S_{e\pi^{-1}(y_0)}}(G_\pi).$$
\end{proof}
\begin{cor} \label{vag-structure}
Let $(X,\sigma)$ be a minimal unique singular orbit system. Let $y_0$ be a singular point of its maximal equicontinuous factor. If 
$e\Et_{y_0}e$ acts effectively then the virtual automorphism group is given by  
$$V(X)\cong C_{S_{e\pi^{-1}(y_0)}}(G_\pi) \times T.$$
\end{cor}
\begin{proof}
Combine Lemmata~\ref{lem-aus1} and \ref{lem-vag-structure}. \end{proof}
We provide a criterion for effectiveness of the action of $e\Et_{y_0} e$.
\begin{lem}\label{effective-lemma}
Let $X$ be a set and 
$\Gamma$ be a non-trivial normal subgroup of a subgroup $G\subseteq S_X$ which acts transitively on $X$. Then  $\Gamma$ acts effectively on $X$.
\end{lem}
\begin{proof}
Let $F_\Gamma=\{ a\in X:\Gamma(a)=\{a\}  \}$. 
If $h\in N_{S_X}( \Gamma)$ then for $a\in F_\Gamma$
\[  \Gamma h(a)  = \{ \gamma h(a): \gamma\in \Gamma    \} =  \{ h\gamma' (a): \gamma'\in \Gamma    \} = \{ h(a)\},  \]
so that $h(a)\in F_\Gamma$. 
By assumption, $G$ lies in $N_{S_X}( \Gamma)$, so that $G(F_{\Gamma})\subseteq F_{\Gamma}$. Since $G$ is transitive we have either $F_\Gamma=X$ or $F_\Gamma=\emptyset$. In the first case, $\Gamma$ can consist only of the identity, and in the second $\Gamma$ acts effectively.
\end{proof}
We thus see that $e\Et_{y_0} e$ acts effectively if it is non-trivial and 
$e\Ef_{y_0} e$ acts transitively on $e\pi^{-1}(y_0)$.

\subsubsection{Bijective substitutions}
We now focus again on the dynamical systems of primitive aperiodic bijective substitutions.
Recall that in this case we can identify $e\pi^{-1}(0)$, the image of the singular fibre at $0\in\Z_\ell$ under the chosen minimal idempotent $e\in E(X_\theta,\Z^+)$, with the alphabet $\Aa$ using the map $\evo^0$. Under this isomorphism   
$G_\pi = e \Ef_0 e \cong G_\theta$ and $e \Et_0 e \cong \Glstr$.  Since $\theta$ is assumed primitive, the structure group $\Gstr$ must act transitively on $\Aa$. Aperiodicity of the substitution implies that $\rset$ must consist of at least $2$ elements. 
Hence $\Gamma_\theta$ is non-trivial, so by Lemma \ref{effective-lemma}, $\Glstr$   acts effectively on $\mathcal A$.

\begin{cor} \label{vag-structure-height}
Let $\theta$ be a primitive  aperiodic bijective substitution.
The virtual automorphism group is given by  
$$V(X_\theta)\cong C_{S_\Aa}(\Gstr) \times \Z.$$
\end{cor}
\begin{proof} All hypothesis of Corollary~\ref{vag-structure} are satisfied. \end{proof}

\begin{cor}\label{Vag=Aut} The dynamical system of a primitive aperiodic bijective substitution is semi-regular.
\end{cor}

\begin{proof}
We see from Theorems  ~\ref{automorphism-theorem} and Corollary~\ref{vag-structure-height} 
that the virtual automorphism group is isomorphic to the automorphism group. 
Furthermore, their fibre preserving parts are isomorphic. Since these are finite groups and the automorphism group is included in the virtual automorphism group, the map from Definition~\ref{semi-regular-definition} must be an isomorphism.
\end{proof}
Since the virtual automorphism group $V(X_\theta)$ can be expressed by means of the Ellis semigroup $E(X_\theta)$, as we saw above, Corollary~\ref{Vag=Aut} describes the relation between $E(X_\theta)$ and $\Aut(X_\theta)$. 
Note that no non-trivial element of $\Aut^{fib}(X_\theta)$ can be an element of $E^{fib}(X_\theta)$ as 
$E^{fib}(X_\theta)\backslash \{ \Id\}$ is a proper ideal and therefore cannot contain an invertible element.

\section{Examples}\label{Examples}

We provide here a list of examples of Ellis semigroups of dynamical systems defined by a primitive, aperiodic, bijective substitution $\theta$ of constant length $\ell$ over a finite alphabet $\Aa$. For the benefit of the reader we summarise results and recall calculation of $E(X_\theta)$.

$E(X_\theta)$ is the disjoint union of its kernel $\M(X_\theta)$ with the acting group $\Z$.
It depends only on the $\Rr$-set $\rset\subset S_\Aa$, which can easily be computed  as in Lemma \ref{I-is-everything}.
But to describe its associated Rees matrix form we make a choice of minimal idempotent $e\in  E(X_\theta,\Z^+)$ which amounts to a choice of element $g_0\in \rset$. Different choices for $e$ lead to isomorphic expressions, and, as far as the fibre preserving parts are concerned, even homeomorphic ones.

The first result is that the structural semigroup $\Sfib$ is isomorphic to the normalised matrix semigroup
$$\Sfib \cong M[\Gstr;I_\theta,\{\pm\};A]$$
where the structure group $\Gstr$  is the group generated by $\rset$, the $+$ entries of $A$ equal 1, and the $-$ entries of  $A$ equal $g_0^{-1} g$, $g\in \rset$. They generate a subgroup of $\Gstr$ which we call the little structure group $\Gamma_\theta$.
Everything is finite at this level and so topologically trivial.

Next, the fibre preserving part $\M^{fib}(X_\theta)$ is topologically isomorphic the normalised matrix semigroup  
$$\M^{fib}(X_\theta)\cong M[\RSTf;I_\theta,\{\pm\};A]$$
where $\rset$ is the same as above. The quotient $\Gstr/\Glstr$ of $\Gstr$ by the normal completion $\Glstr$ of the little structure group must be a cyclic group, its order $h$ is the {generalised} height of the substitution.
The structure group of $\M^{fib}(X_\theta)$ is $\Z/h\Z$-graded and its subgroup of elements of degree $0$ is
$$\RSTf_0 \cong \Glstr^{\Z_\ell/\Z}.$$
If $\Gstr$ contains an element of order $h$ then 
$$\RSTf \cong \Glstr^{\Z_\ell/\Z}\rtimes \Z/h\Z,$$
a semidirect product whose explicit expression depends on the choice of 
an element $\mathfrak f$ of $\RSTf$ of degree $1$.
The sandwich matrix $A$ is the same as that for $\Sfib$, because we view $\Glstr$ as a subgroup of $\Glstr^{\Z_\ell/\Z}$: an element $g\in \Glstr$ maps to the function of 
$\Glstr^{\Z_\ell/\Z}$ which takes value $g$ on $[0]$ and $\one$ otherwise. 
If generalised height is trivial then we can rewrite the above 
$$\M^{fib}(X_\theta)\cong \Sfib\times \prod_{[z]\in\Z_\ell/\Z\atop [z]\neq [0]} \Gstr$$
again a topological isomorphism.

Finally,  the kernel $\M(X_\theta)$ of the full Ellis semigroup has Rees-matrix form
\begin{equation}\label{eq-Rees-for-M}\nonumber
\M(X_\theta) \cong M[\RST;I_\theta,\{\pm\};A],
\end{equation}
where $\RST$ is an extension of the equicontinuous factor $\Z_\ell$ by $\RSTf$.
This isomorphism is only algebraic. Again $\rset$ and the sandwich matrix $A$ are the same,  as $\RSTf$ is a subgroup of $\RST$. The extension is algebraically split if the height is equal to the classical height $h=h_{cl}$ of the substitution. In this case
$$\RST \cong \RSTf \rtimes \Z_\ell,$$
algebraically.
\bigskip

Besides the details for the Ellis semigroup we provide below also $C_\Aa(\Gstr)$, the centraliser of $\Gstr$ in the group of permutations of the alphabet, which is also
$\mbox{\rm Aut}^{fib}(X_\theta)$.

For arbitrary size of the alphabet we can say the following.  There is no aperiodic bijective substitution with $|\rset|=1$.
If $\rset$ contains two elements then $\Gamma_\theta$ must be a cyclic subgroup of $\Gstr$.  

\subsection{Two-letter alphabet} To be compatible with primitivity and aperiodicity we must have   $\rset=S_2$. Hence $\Gstr=S_2=\Z/2\Z$ and $\Gamma_\theta=\Glstr=S_2$. Thus all primitive, aperiodic, bijective substitutions on a two-letter alphabet have the same structural semigroup, namely  $M[ S_2; S_2, \{ \pm\}; A]$.
The sandwich matrix is $A=\begin{pmatrix}  \one & \one \\ \one & \omega\end{pmatrix}$ where $\omega$ interchanges $a$ with $b$.
The generalised height is trivial for these substitutions. We thus have
$$\RSTf \cong S_2^{\Z_\ell/\Z},\quad \mathrm{and}\quad
\RST \cong S_2^{\Z_\ell/\Z} \rtimes \Z_\ell$$
where $\ell$ is the length of the substitution.

$M[S_2; S_2; \{ \pm\}; A]$ is perhaps the simplest non-orthodox semigroup. 
Since $S_2$ is abelian, we have $C_\Aa(\Gstr)=  \Gstr$. Thus all these substitutions have $\mbox{\rm Aut}^{fib}(X_\theta) = S_2$, generated by the map $\omega$. The simplest example of this type is the (simplified) Thue-Morse substitution,
$\theta(a) = abba$, $\theta(b) = baab$, where the above result has been obtained by Marcy Barge in a direct calculation \cite{Barge}.  

\subsection{Three-letter alphabet}\label{ex462} 
If $\Gstr$ is a subgroup of $S_2\subset S_3$ then we reproduce the above results for the semigroup, but these can never be realised by a primitive substitution on three letters, as one letter would stay fixed. So we consider the two possible other cases,  $\Gstr=S_3$ and $\Gstr=A_3\cong \Z/3\Z$. For
$\Gstr=S_3$, we give below examples where $\Gamma_\theta \cong \Z/2\Z$ or $\Gamma_\theta\cong\Z/3\Z$. In the first case, $\Gamma_\theta \cong \Z/2\Z$ has normal completion $\Gstr$ and so the height is trivial. In the second case  $\Gamma_\theta\cong\Z/3\Z$ is normal in $\Gstr$ and the height equal to $2$. The example we provide for this case has classical height $1$. We also give an example where  $\Gstr\cong \Z/3\Z$, which, for aperiodic $\theta$, forces $\Gamma_\theta\cong\Z/3\Z$ so that we have again trivial height. 
\begin{enumerate}
\item
Consider the substitution $\theta$
\[
\begin{array}{c} a\\ b\\ c \end{array} 
\mapsto
\begin{array}{c} a\\ b\\ c \end{array}
\!\!\!\!\!\!{\begin{array}{c} b\\ a\\ c \end{array}}
\!\!\!\!\!\!{\begin{array}{c} c\\ b\\ a \end{array}}
\!\!\!\!\!\!{\begin{array}{c} c\\ a\\ b \end{array}}
\!\!\!\!\!\!{\begin{array}{c} a\\ b\\ c \end{array}}
\]
Then it can be verified that $I_{\theta}=\left\{    \begin{pmatrix} b\\a\\c \end{pmatrix},  \begin{pmatrix} b\\c\\a \end{pmatrix}     \right\}$ which generates $\Gstr = S_3$.  
The structural semigroup is  $M[S_3; I_\theta, \{ \pm\}; A]$, whose normalised sandwich matrix $A=\begin{pmatrix}  \one & \one  \\  
\one &\omega \end{pmatrix}$, where $\omega$ exchanges $b$ with $c$.
One finds $\Gamma_\theta=\left\{ \one,  \begin{pmatrix} c\\b\\a \end{pmatrix}  \right\}\cong \Z/2\Z$, which is not normal in $S_3$, and $ \Glstr = G_\theta= S_3$. Thus $\theta$ has trivial generalised height. Hence
$$\RSTf \cong S_3^{\Z_5/\Z},\quad \mathrm{and}\quad
\RST \cong S_3^{\Z_5/\Z} \rtimes \Z_5$$
Also $C_\Aa(\Gstr) = \mbox{\rm Aut}^{fib}(X_\theta)$ is trivial.

\item

Consider the substitution $\theta$
\[
\begin{array}{c} a\\ b\\ c \end{array} 
\mapsto
\begin{array}{c} a\\ b\\ c \end{array}
\!\!\!\!\!\!{\begin{array}{c} b\\ a\\ c \end{array}}
\!\!\!\!\!\!{\begin{array}{c} a\\ b\\ c \end{array}}
\!\!\!\!\!\!{\begin{array}{c} c\\ b\\ a \end{array}}
\!\!\!\!\!\!{\begin{array}{c} a\\ b\\ c \end{array}}
\!\!\!\!\!\!{\begin{array}{c} a\\ c\\ b \end{array}}
\!\!\!\!\!\!{\begin{array}{c} a\\ b\\ c \end{array}}
\]
It has $\theta_0=\theta_2=\theta_4=\theta_6 = \id$ and the other three are the transpositions of $S_3$,
$\theta_1 = \begin{pmatrix} b\\a\\c\end{pmatrix}$,
$\theta_3 =  \begin{pmatrix}c\\b\\a \end{pmatrix}$, and 
$\theta_5 =\begin{pmatrix}a\\c\\b \end{pmatrix}$.
Hence $\rset = \{\theta_1,\theta_3,\theta_5\}$ and $\Gstr=S_3$. The structural semigroup is $M[S_3; I_\theta, \{ \pm\}; A]$ has  normalised sandwich matrix 
$\begin{pmatrix}  \one & \one & \one \\  
\one &\omega&\omega^2 \end{pmatrix}$, where 
$\omega =  \begin{pmatrix} b\\c\\a\end{pmatrix}$, a cyclic permutation.

Every element in $\Gamma_\theta$ is an even permutation, thus $\Gamma_\theta =
\Glstr=A_3$. It follows that $\Gstr/\Gamma_\theta \cong \Z/2\Z$ and $\theta$ has generalised height equal to 2 and therefore
$$\RSTf_0 \cong {A_3}^{\Z_7/\Z},\quad \RSTf \cong {A_3}^{\Z_7/\Z}\rtimes \Z/2\Z,$$
as $\theta_1\in \rset$ has order $2$.

Note that the substitution has trivial classical height as any fixed point must contain the word $aa$. We do therefore not know whether the extension ${A_3}^{\Z_7/\Z}\rtimes \Z/2\Z\hookrightarrow \RST \stackrel{\tilde\pi}\twoheadrightarrow \Z_7$ defining the structure group $\RST$ of $\M(X_\theta)$ splits.
Again, $C_\Aa(\Gstr)=\mbox{\rm Aut}^{fib}(X_\theta)$ is trivial.

This example has a natural generalisation to alphabets of any size $s$ with $\Gstr=S_s$ and $\Gamma_\theta=A_s$, the alternating group on $s$ elements.

\item
Consider the substitution $\theta$
$$
\begin{array}{c c l}
a &  & abc \\
b & \mapsto & bca \\
c &  & cab
\end{array}
$$
whose third power is simplified. We find $\rset = \{\one,\omega,\omega^{2}\}\cong \Z/3\Z$ where $\omega =\begin{pmatrix} b\\c\\a \end{pmatrix}$ 
is a cyclic permutation. It follows that $\Gstr=\Gamma_\theta=\Glstr\cong\Z/3\Z$. The structural semigroup is $M[\Z/3\Z; I_\theta, \{ \pm\}; A]$ where
$A= \begin{pmatrix}  \one & \one & \one \\  
\one&\omega&\omega^{2} \end{pmatrix}$.  
$$\RSTf \cong (\Z/3\Z)^{\Z_{3}/\Z},\quad \mathrm{and}\quad
\RST \cong (\Z/3\Z)^{\Z_{3}/\Z} \rtimes \Z_3.$$
Finally, $C_\Aa(\Gstr)=\mbox{\rm Aut}^{fib}(X_\theta)\cong\Z/3\Z$, generated by $\omega$.
There is an obvious generalisation of this example to alphabets of any size $s\geq 2$, the case $s=2$ corresponding again to the Thue-Morse substitution.

\end{enumerate}

\subsection{Four-letter alphabet}
Our last example is related to the dihedral group $D_4$. Consider the substitution
$\theta$ of length $7$
$$
\begin{array}{c c l}
a &  & abadcba \\
b & \mapsto & badcbab \\
c &  & cdcbadc\\
d & & dcbadcd
\end{array}
$$
which has classical height 2. If we identify the letters with the edges of a square whose center is $0$ in such a way that 
$a$ corresponds to the lower right corner and we order the edges counterclockwise, then
$\theta_1$ and $\theta_3$ amount to the reflection along the $x$-axis and the $y$-axis, resp., while $\theta_2$ and $\theta_4$ amount to the reflection along the diagonal with slope $-1$ and $+1$, resp.. Finally $\theta_5=\theta_1$ and $\theta_6=\theta_0=\one$. 
Thus $\rset = \{\theta_1,\rho\}$ where $\rho=\theta_1\theta_2$ is the rotation by $\frac\pi2$ to the right. It follows that $G_\theta = D_4$ is the dihedral group of order $4$ and that $\Gamma_\theta$ the group of order $2$ generated by the reflection $\theta_2$. Its normal completion is thus the group generated by the reflections $\theta_2$ and $\theta_4$, which commute, so $\Glstr\cong \Z/2\Z\times\Z/2\Z$ showing that height is equal to the classical height, namely 2. Moreover, the element $\theta_1$ of $\rset$ has order $2$.
The structural semigroup is
$\Sfib=M[D_4; I_\theta, \{ \pm\}; A]$  with $A=\begin{pmatrix}  \one & \one  \\  
\one &\theta_4 \end{pmatrix}$.
Furthermore,
$$\RSTf \cong (\Z/2\Z\times\Z/2\Z)^{\Z_7/\Z}\rtimes\Z/2\Z,\quad \mathrm{and}\quad
\RST \cong (\Z/2\Z\times\Z/2\Z)^{\Z_7/\Z}\rtimes\Z/2\Z \rtimes \Z_7$$
Finally
$C_\Aa(\Gstr) = \mbox{\rm Aut}^{fib}(X_\theta)= \{ \id,\rho^2 \}$.

 \section*{Acknowledgement}
 The authors are very grateful to the referee, whose numerous suggestions and questions led not only to an improvement in the exposition, but also to an extension of the results, namely the recovering of the full Ellis semigroup from its fibre preserving part in 
 Corollary~\ref{cor:reproducing M}, and Theorems~\ref{thm-main2} and 
 \ref{thm-main4}.

{\footnotesize
\bibliographystyle{abbrv}


\begin{thebibliography}{10}

\bibitem{Aujogue}
J.-B. Aujogue.
\newblock Ellis enveloping semigroup for almost canonical model sets of an
  euclidean space.
\newblock {\em Algebraic \& Geometric Topology}, 15(4):2195--2237, 2015.

\bibitem{Aujogue-Barge-Kellendonk-Lenz}
J.-B. Aujogue, M.~Barge, J.~Kellendonk, and D.~Lenz.
\newblock Equicontinuous factors, proximality and {E}llis semigroup for
  {D}elone sets.
\newblock In {\em Mathematics of aperiodic order}, volume 309 of {\em Progr.
  Math.}, pages 137--194. Birkh\"{a}user/Springer, Basel, 2015.

\bibitem{Auslander}
J.~Auslander.
\newblock {\em Minimal flows and their extensions}, volume 153 of {\em
  North-Holland Mathematics Studies}.
\newblock North-Holland Publishing Co., Amsterdam, 1988.
\newblock Notas de Matem\'{a}tica [Mathematical Notes], 122.



\bibitem{AG-2019}
J.~Auslander and E.~Glasner.
\newblock{On the virtual automorphism group of a minimal flow}
\newblock to appear in {\em Ergodic Theory and Dynamical Systems}
\newblock DOI={10.1017/etds.2020.8.}

\bibitem{Baake-Grimm}
M.~Baake and U.~Grimm.
\newblock {\em Aperiodic order. {V}ol. 1}, volume 149 of {\em Encyclopedia of
  Mathematics and its Applications}.
\newblock Cambridge University Press, Cambridge, 2013.
\newblock A mathematical invitation, With a foreword by Roger Penrose.

\bibitem{Barge}
M.~Barge.
\newblock private communication.

\bibitem{BargeKellendonk}
M.~Barge. and J.~Kellendonk.
\newblock {\em Complete regularity of Ellis semigroups of $\mathbb Z $-actions.} 
\newblock Preprint arXiv:2001.02751 (2020).

\bibitem{Blanchard}
F.~Blanchard, E.~Glasner, S.~Kolyada, A.~Maass
\newblock {\em On Li-Yorke pairs},
\newblock {Journal fur die reine und angewandte Mathematik}
  {\bf 547}, {51--68}, {2002}.

\bibitem{carruth1983theory}
J.~H. Carruth, J. Hildebrant, and R.~J. Koch.
\newblock The theory of topological semigroups.
\newblock Vol. 75.  Marcel Dekker Incorporated, 1983.

\bibitem{coven-quas-yassawi}
E.~M. Coven, A.~Quas, and R.~Yassawi.
\newblock Computing automorphism groups of shifts using atypical equivalence
  classes.
\newblock {\em Discrete Anal.}, pages Paper No. 3, 28, 2016.

\bibitem{dekking}
F.~M. Dekking.
\newblock The spectrum of dynamical systems arising from substitutions of
  constant length.
\newblock {\em Z. Wahrscheinlichkeitstheorie und Verw. Gebiete},
  41(3):221--239, 1977/78.

\bibitem{Ellis-1960}
    R.~Ellis.
     \newblock A semigroup associated with a transformation group.
   \newblock Trans. Amer. Math. Soc., 
 94:272--281, 1960.
      
    
\bibitem{Pytheas-Fogg}
N.~P. Fogg.
\newblock {\em Substitutions in dynamics, arithmetics and combinatorics},
  volume 1794 of {\em Lecture Notes in Mathematics}.
\newblock Springer-Verlag, Berlin, 2002.
\newblock Edited by V. Berth\'{e}, S. Ferenczi, C. Mauduit and A. Siegel.

\bibitem{fuhrmann2018irregular}
G.~Fuhrmann, E.~Glasner, T.~J{\"a}ger, and C.~Oertel.
\newblock Irregular model sets and tame dynamics.
\newblock {\em arXiv preprint arXiv:1811.06283}, 2018.

\bibitem{furstenberg1989idempotents}
H.~Furstenberg and Y.~Katznelson.
\newblock Idempotents in compact semigroups and ramsey theory.
\newblock {\em Israel Journal of Mathematics}, 68(3):257--270, 1989.


\bibitem{glasnerCM}
E.~Glasner.
\newblock On tame dynamical systems.
\newblock {\em Colloquium Mathematicum}, 105(2):284--295, 2006.

\bibitem{glasner2018structure}
E.~Glasner.
\newblock The structure of tame minimal dynamical systems for general groups.
\newblock {\em Inventiones mathematicae}, 211(1):213--244, 2018.

\bibitem{hindman}
N.~Hindman and D.~Strauss.
\newblock Algebra in the Stone-Cech compactification: theory and applications.
\newblock Walter de Gruyter, 2011.


\bibitem{howie1995fundamentals}
J.~M. Howie.
\newblock {\em Fundamentals of semigroup theory}, volume~12.
\newblock Clarendon Oxford, 1995.



\bibitem{Huang-2006}
    W. Huang.
    \newblock
     Tame systems and scrambled pairs under an abelian group
              action.
   \newblock Ergodic Theory Dynam. Systems,
 26(5):1549--1567, 2006.
     



\bibitem{kamae}
T.~Kamae.
\newblock A topological invariant of substitution minimal sets.
\newblock {\em J. Math. Soc. Japan}, 24:285--306, 1972.

\bibitem{L-M}
M.~Lema\'{n}czyk and M.~K. Mentzen.
\newblock On metric properties of substitutions.
\newblock {\em Compositio Math.}, 65(3):241--263, 1988.



\bibitem{lm}
{D.~Lind and B.~Marcus}.
\newblock {An introduction to symbolic dynamics and coding}.
\newblock {Cambridge University Press}, {1995}

\bibitem{martin}
J.~C. Martin.
\newblock Substitution minimal flows.
\newblock {\em Amer. J. Math.}, 93:503--526, 1971.


\bibitem{M-Y}
C.~Muellner and R.~Yassawi.
\newblock Automorphisms of automatic shifts. \newblock To appear in{\em Ergodic Theory and Dynamical Systems}
\newblock
	DOI={10.1017/etds.2020.13}.

\bibitem{petrich1999completely}
M.~Petrich and N.~R. Reilly.
\newblock {\em Completely regular semigroups}, volume~27.
\newblock John Wiley \& Sons, 1999.

\bibitem{Staynova}
P.~{Staynova}.
\newblock {The Ellis Semigroup of a Generalised Morse System}.
\newblock To appear in {\em Ergodic Theory and Dynamical Systems}.
\newblock DOI={10.1017/etds.2019.75}.
	

\end{thebibliography}

\def\ocirc#1{\ifmmode\setbox0=\hbox{$#1$}\dimen0=\ht0 \advance\dimen0
  by1pt\rlap{\hbox to\wd0{\hss\raise\dimen0
  \hbox{\hskip.2em$\scriptscriptstyle\circ$}\hss}}#1\else {\accent"17 #1}\fi}

}

\end{document}